\documentclass{amsart}

\usepackage{amsmath,amssymb,amsfonts,amsthm}
\usepackage{enumerate}
\usepackage{mathtools}
\usepackage{stmaryrd}
\usepackage{centernot}
\usepackage{float}
\usepackage{tabularx}
\usepackage{tikz}
\usetikzlibrary{arrows.meta,matrix}
\usepackage{tikz-cd}
\usepackage[mathscr]{euscript}
\usepackage{microtype}
\usepackage{xcolor}
\definecolor{mygray}{gray}{0.85}
\usepackage[linecolor=black,backgroundcolor=mygray,colorinlistoftodos,prependcaption,textsize=small]{todonotes}
\usepackage[colorlinks,citecolor=blue,urlcolor=blue,linkcolor=blue]{hyperref}

\usepackage[backgroundcolor=mygray,colorinlistoftodos,prependcaption,textsize=small]{todonotes}
\usepackage{xcolor}

\renewcommand{\leq}{\leqslant}
\renewcommand{\geq}{\geqslant}

\newcommand{\mrm}[1]{\mathrm{#1}}

\makeatletter
\def\subsection{\@startsection{subsection}{3}%
 \z@{.5\linespacing\@plus.7\linespacing}{.3\linespacing}%
 {\bfseries\centering}}
\makeatother

\makeatletter
\def\subsubsection{\@startsection{subsubsection}{3}%
 \z@{.5\linespacing\@plus.7\linespacing}{.3\linespacing}%
 {\centering}}
\makeatother


\makeatletter
\def\myfnt{\ifx\protect\@typeset@protect\expandafter\footnote\else\expandafter\@gobble\fi}
\makeatother

\newcommand{\pureindep}[1][]{%
 \mathrel{
  \mathop{
   \vcenter{
    \hbox{\oalign{\noalign{\kern-.3ex}\hfil$\vert$\hfil\cr
       \noalign{\kern-.7ex}
       $\smile$\cr\noalign{\kern-.3ex}}}
   }
  }\displaylimits_{#1}
 }
}

\newcommand{\indep}[2]{%
 \mathrel{
  \mathop{
   \vcenter{
    \hbox{%
\oalign{
\noalign{\kern-.3ex}\hfil$\vert$\hfil\cr
       \noalign{\kern-.7ex}
       $\smile$\cr\noalign{\kern-.3ex}
}
}
   }
}^{\!\!\!\!\!#2}_{\!\!\hspace{-0.1em}#1}
 }
}

\newcommand{\displayindep}[2]{%
 \mathrel{
  \mathop{
   \vcenter{
    \hbox{%
\oalign{
\noalign{\kern-.3ex}\hfil$\vert$\hfil\cr
       \noalign{\kern-.7ex}
       $\smile$\cr\noalign{\kern-.3ex}
}
}
   }
}^{\!\!\hspace{-0.1em}#2}_{\!\!\hspace{-0.1em}#1}
 }
}

\newcommand{\displayfindep}[2]{%
 \mathrel{
  \mathop{
   \vcenter{
    \hbox{%
\oalign{
\noalign{\kern-.3ex}\hfil$\vert$\hfil\cr
       \noalign{\kern-.7ex}
       $\smile$\cr\noalign{\kern-.3ex}
}
}
   }
}^{\!\hspace{-0.14em}#2}_{\!\!\hspace{-0.05em}#1}
 }
}

\renewcommand{\prec}{\preccurlyeq}

\newcommand{\gS}{\mathcal{S}}

\newtheoremstyle{plainupright}
 {3pt}  
 {3pt}  
 {\normalfont}  
 {}   
 {\bfseries}   
 {.}   
 {.5em} 
 {}   

\theoremstyle{plainupright}
\newtheorem{theorem}{Theorem}[section]

\newtheorem{corollary}[theorem]{Corollary}

\newtheorem{lemma}[theorem]{Lemma}
\newtheorem{proposition}[theorem]{Proposition}
\newtheorem{example}[theorem]{Example}

\newtheorem{question}[theorem]{Question}

\theoremstyle{definition}

\newtheorem{fact}[theorem]{Fact}
\newtheorem{definition}[theorem]{Definition}
\newtheorem{remark}[theorem]{Remark}

\newcounter{claimcounter}
\numberwithin{claimcounter}{theorem}

\newcommand{\pleq}{\leq_{\mrm{pp}}}
\newcommand{\pleqq}{\leq^{\mrm{weak}}_{\mrm{pp, \infty}}}
\newcommand{\oleq}{\leq_{\oplus}}

\newcommand{\rmod}{R\text{-}\mathrm{Mod}}
\newcommand{\zmod}{\mathbb{Z}\text{-}\mathrm{Mod}}

\newcommand{\gtp}{\mathbf{gtp}}

\begin{document}
\begin{abstract}
Let $R$ be a ring. We organize the abstract elementary classes whose underlying class is the class of all $R$-modules and whose strong submodel relation lies between the submodule and direct summand relations into a lattice $\mathscr{L}_{R}$, ordered by reverse inclusion. We establish the basic lattice-theoretic properties of $\mathscr{L}_{R}$ and investigate its two natural sublattices, below and above purity. Below purity, we isolate relations defined by first-order $\mrm{pp}$-formulas for which amalgamation, tameness, and stability hold. Above purity, we introduce relations defined by infinitary $\mrm{pp}$-formulas and prove a broad stability result. Specializing to abelian groups, we show that the lattice $\mathscr{L}_{\mathbb{Z}}$ has the following properties: it  has a strong submodel relation that is not positive syntactic, it contains an uncountable antichain and a strictly increasing proper-class-sized chain, and it has a broad region above purity where amalgamation fails.
\end{abstract}

\title[The lattice of AECs of modules]{The lattice of abstract elementary classes of modules}

\thanks{Research of Marcos Mazari-Armida was partially supported by NSF grant DMS-2348881 and Simons Foundation grant MPS-TSM-00007597. Research of Gianluca Paolini was supported by project PRIN 2022 ``Models, Sets and Classifications'', prot. 2022TECZJA, and by INdAM Project 2024 (Consolidator grant) ``Groups, Crystals and Classifications''.}


\author{Tapani Hyttinen}
\address{Department of Mathematics and Statistics, University of Helsinki,
P.O. Box 68 (Pietari Kalmin katu 5), FI-00014 University of Helsinki, Finland.}
\email{tapani.hyttinen@helsinki.fi}

\author{Marcos Mazari-Armida}
\address{Department of Mathematics,
	Baylor University,
	Sid Richardson Building,
	1410 S.~4th Street,
	Waco, TX 76706, USA}
\email{marcos\_mazari@baylor.edu}
\urladdr{https://sites.baylor.edu/marcos\_mazari/}

\author{Gianluca Paolini}
\address{Department of Mathematics ``Giuseppe Peano'', University of Turin,
Palazzo Campana, Via Carlo Alberto 10, 10123 Torino, Italy.}
\email{gianluca.paolini@unito.it}
\urladdr{https://sites.google.com/view/gianlucapaolini/}
\raggedbottom
\maketitle

\setcounter{tocdepth}{2}
\tableofcontents
\newpage
\flushbottom


\section{Introduction}

Abstract elementary classes (AECs for short) constitute the prominent model-theoretic framework used to study non-elementary classes of structures. These were introduced in the seventies by Shelah  \cite{sh88}. For many years the focus was on developing the abstract theory  \cite{baldwinbook09, bova-tame, shelahaecbook}, but recently there has been significant activity on finding interactions and applications of AECs to module theory \cite{bet, bon25,  maztor, mj, 1121, satr}. This paper introduces a new kind of interaction between AECs and module theory which we expect  will be very fruitful.

An abstract elementary class is a pair $\mathcal{K} = (\mathbf{K}, \prec)$, where $\mathbf{K}$ is a class of structures and $\prec$ is a partial order on $\mathbf{K}$, subject to a few natural axioms (see Definition \ref{def_AEC}). Two of the key axioms are that $\mathcal{K}$ is closed under directed colimits and that every element of $\mathbf{K}$ can be decomposed into \emph{small} elements of $\mathbf{K}$. In this paper, $\mathbf{K}$ will always be a class of modules over a fixed ring $R$; in most cases it will in fact be the class of \emph{all} $R$-modules.

Most of the interactions of AECs and module theory have taken place by fixing a class of modules and studying the AEC one obtains by letting the  partial order on the class be {\em the pure submodule relation}\footnote{$A$ is a pure submodule of $B$ if every finite system of $R$-linear equations with the right-hand side in $A$ that is solvable in $B$ is already solvable in $A$.}. The choice of purity as the \emph{de facto} partial order to study classes of modules is motivated both by the algebraic importance of purity \cite{fuchs, GT} and the first-order model-theoretic result of $\mrm{pp}$-quantifier elimination \cite{prest_book1}. Nevertheless there is no deep reason why one should restrict to purity and recently interesting results for other partial orders have been obtained \cite{cox, mj, mt1, 1121}.

In this paper, we fix the class of all modules and instead of studying the AEC one obtains by fixing the partial order to be purity  (as it was done in \cite{maz1}) or another partial order, we study all the possible partial orders that give an AEC on the class of all  modules.

This new approach has two key benefits. The first is that it allows us to obtain results for a family of AECs instead of a single AEC. This lets us recover many results from the literature (see Remarks \ref{gen-sta-1} and \ref{gen-sta-2}). The second, which we find to be the most significant, is that it allows us to uncover a new mathematical object $\mathscr{L}_{R}$, which is the main object of study of this paper. $\mathscr{L}_{R}$ is the lattice of  binary relations $\preccurlyeq$ on the class of $R$-modules such that the class of $R$-modules equipped with $\preccurlyeq$ is an abstract elementary class and $\preccurlyeq$ lies between the submodule and direct summand relations, ordered by reverse inclusion (see Definition \ref{main_def} and Corollary \ref{lattice}). Since the underlying class  is fixed and only the strong submodel relation varies, $\mathscr{L}_{R}$ measures how different algebraically natural notions of submodule change the model-theoretic behavior of the class of modules.

 The results on the lattice are of two sorts. The first consists of isolating properties that allow us to show that a broad region of the lattice has or fails to have important model-theoretic properties. The second consists of proving lattice-theoretic results. These results depend heavily on the structure of the ring $R$; for instance, the lattice is trivial if $R$ is semisimple. Thus, in this paper we focus mostly on the case when $R=\mathbb{Z}$.

We divide the study of the lattice into  two natural sublattices: the sublattice below purity $\mathscr{L}^1_{R}$ and the sublattice above purity $\mathscr{L}^2_{R}$. The techniques we employ to study $\mathscr{L}^1_{R}$ mimic those used for purity, while the techniques we employ for $\mathscr{L}^2_{R}$ are new. Regarding  $\mathscr{L}^2_{R}$, it is worth noticing that even the existence of elements in $\mathscr{L}^2_{R}$ strictly extending purity was not known before this work. We construct a plethora of elements in $\mathscr{L}^2_{R}$ by using \emph{infinitary $\mrm{pp}$-formulas} (see Definition \ref{deflambdaformulas}).

The relations above purity, in $\mathscr{L}^2_{R}$, measure the extent to which solvability of infinite systems of linear equations in the ambient module is reflected in the submodule. The relation $\pleqq$ requires the two modules to agree on every weak infinitary $\mrm{pp}$-formula; equivalently, whenever a possibly infinite family of first-order $\mrm{pp}$-formulas over a finite tuple from the smaller module is realized by a finite tuple in the larger module, it is already realized in the smaller one. The relations $\leq^{(\kappa,\alpha)}_{\mrm{pp}}$ and $\leq^{(\kappa,\infty)}_{\mrm{pp}}$ refine the same idea by controlling the size and the syntactic complexity of the admissible infinitary conjunctions (see Definition \ref{or_infty}). This connects with the study of $\mu$-$\mrm{AEC}$s of modules initiated in \cite{carnevale}. The relations of $\mu$-purity considered there, however, fail to give rise to $\mrm{AEC}$s, so the relations $\pleqq$, $\leq^{(\kappa,\alpha)}_{\mrm{pp}}$ and $\leq^{(\kappa,\infty)}_{\mrm{pp}}$ introduced here take on particular interest as they are genuinely new.
 
The main model-theoretic properties studied in this paper are stability and the amalgamation property. \emph{Stability} is one of the key dividing lines in model theory (see Definition \ref{def_stability}); it was introduced for elementary classes in \cite{sh1} and for AECs in \cite{sh394}. We show that stability is pervasive in both $\mathscr{L}^1_{R}$ and $\mathscr{L}^2_{R}$. In passing, we recall that it was an open problem whether there exists an \emph{unstable} AEC of modules with respect to purity; this was answered in the positive in \cite{1121} (see also \cite{mazaripaolini}).
 
 \begin{theorem}\
 \begin{enumerate}[(1)]
 \item (Lemma \ref{stable-(I)}) Assume that $\preccurlyeq \in  \mathscr{L}^1_{R}$ is pleasant (see Definition \ref{def_pleasant}). If $\lambda^{\operatorname{card}(R) + \aleph_0}= \lambda$, then $(\rmod, \preccurlyeq)$ is $\lambda$-stable.
 \item (Theorem \ref{sta-nice}) Assume that $\preccurlyeq \in \mathscr{L}^2_{R}$ is E-nice  (see Definition \ref{pre_def_nice}). Then $(\rmod, \preccurlyeq)$ is stable.
 \end{enumerate}
 \end{theorem}
 
 The proofs of the above two results are very different. The first result is obtained by first syntactically characterizing (Galois or orbital) types (see Theorem \ref{char-gtps}) and then counting the number of types using similar tools to the classical tools from first-order model theory. The second result uses in a key way the pushout construction and a local character-like result (see Claim 1 of Theorem \ref{sta-nice}).

\emph{Amalgamation}  is a classical assumption on AECs (see Definition \ref{amal}). This is the case because amalgamation is seen as a weak consequence of the compactness theorem, and because it is expected to follow from categoricity in a tail of cardinals \cite[Conjecture 2.3]{grossberg2002}. We show that amalgamation holds often in $\mathscr{L}^1_{R}$, but often fails in $\mathscr{L}^2_{R}$ when $R = \mathbb{Z}$.

 \begin{theorem}\label{AP_t}\
 \begin{enumerate}[(1)]
 \item (Lemma \ref{AP-(I)})
 Assume $\preccurlyeq \in  \mathscr{L}^1_{R}$ is pleasant. Then $(\rmod,\preccurlyeq)$ has the amalgamation property.
 \item (Theorem \ref{no_AP_th}) Assume $\preccurlyeq\in\mathscr{L}^2_{\mathbb Z}$ is positive syntactic (see Definition \ref{def_syntactic}) and lies above $\pleqq$ (see Definition \ref{def_weak_pp}) in $\mathscr{L}_{\mathbb Z}$. Then $(\zmod,\preccurlyeq)$ does not have the amalgamation property.
 \end{enumerate}
 \end{theorem}
 
 The proofs of the above two results rely heavily on the pushout construction. The second result once again suggests  \cite{mj, sh820} that amalgamation is not as pervasive as commonly assumed. Furthermore, the second result shows that the model-theoretic behavior of the class of modules for a relation above purity is distinct from that for purity.

We obtain a  couple of basic lattice-theoretic properties for $\mathscr{L}_{R}$ for an arbitrary ring $R$ (see Section \ref{b-results}). As mentioned earlier, since  $\mathscr{L}_{R}$ depends so heavily on the ring-theoretic properties of $R$, in this paper we focus on the case when $R =\mathbb{Z}$. The results for $\mathscr{L}_{\mathbb{Z}}$ use classical tools from abelian group theory such as the notion of isotype subgroup.

\begin{theorem}\
\begin{enumerate}[(1)]
\item (Proposition \ref{basic_L1}) $|\mathscr{L}^1_{\mathbb Z}| \leq 2^{2^{\aleph_0}}$.
\item (Lemma \ref{antichain}) $\mathscr{L}^1_{\mathbb Z}$ has an uncountable antichain.
\item (Theorem \ref{class_size}) There is a strictly increasing proper-class-sized chain in $\mathscr{L}_{\mathbb Z}$. In particular, the lattice $\mathscr{L}_{\mathbb Z}$ is a proper class. 
\end{enumerate}
\end{theorem}

 We encourage the reader to look at Appendix \ref{app-A} where we give a depiction of~$\mathscr{L}_{\mathbb{Z}}$.

 Finally, a surprising result is that there is a relation $\prec \in \mathscr{L}_{\mathbb Z}$ below purity  such that $\prec$  cannot be characterized by infinitary $\mrm{pp}$-formulas (see Theorem \ref{n-syntac}). This result is important as it 
 shows that the study of $\mathscr{L}_{\mathbb Z}$ is not simply the study of binary relations generated by infinitary $\mrm{pp}$-formulas.

Taken together, these results reveal a global contrast: stability is widespread both below and above purity, while amalgamation fails on a broad region above purity, and already for abelian groups the resulting lattice is a proper class.

The paper is divided into five sections. Section 2 has some preliminaries and basic results, including an introduction to infinitary $\mrm{pp}$-formulas. Section 3 introduces the main object of study of this paper and presents some of its basic lattice-theoretic properties. Section 4 has important model-theoretic results for broad regions of the sublattices $\mathscr{L}^1_{R}$ and $\mathscr{L}^2_{R}$. Section 5 contains several results on $\mathscr{L}_{\mathbb{Z}}$.

\section{Preliminaries and basic results}

In this section, we present a quick introduction to both the model theory of modules  and abstract elementary classes.

\subsection{Model theory of modules}

All rings considered in this paper are associative with identity, and all modules are left $R$-modules unless stated otherwise. For a fixed ring $R$, we denote the class of $R$-modules by $\rmod$. For $A,B\in\rmod$, we write $A\leq B$ if \emph{$A$ is a submodule of $B$}, and $A\leq_\oplus B$ if \emph{$A$ is a direct summand of $B$}. We write $\operatorname{card}(R)$ for the cardinality of the ring  $R$.

Throughout this subsection, $R$ is a fixed ring. The language of $R$-modules is
$\tau_R=\{0,+,-\}\cup\{r\cdot\mid r\in R\},$
where for each $r\in R$, the symbol $r\cdot$ is interpreted as scalar multiplication by $r$.

Following \cite[Definition 2.2]{shelah}, we introduce \emph{infinitary $\mrm{pp}$-formulas}.

\begin{definition}
\label{deflambdaformulas}
Let $\lambda\leq \kappa$ be infinite cardinals. For every ordinal $\epsilon<\lambda$, we define by induction on $\alpha$ a set $\Lambda_{\alpha,\epsilon}^{\kappa,\lambda}$ of formulas $\varphi(\bar x)$ in $\mathfrak{L}_{\kappa, \lambda}(\tau_R)$ with $|\bar x| = \epsilon$ as follows:
\begin{enumerate}[(1)]
\item For $\alpha = 0$, we define $\Lambda_{\alpha, \epsilon}^{\kappa, \lambda}$ as the set of all equations of the form
\[
\sum_{\zeta < \epsilon} r_\zeta x_\zeta = 0
\]
with $r_\zeta \in R$ and  $r_\zeta\neq 0$ for only finitely many $\zeta<\epsilon$.
\item For $\alpha$ a limit ordinal, we define $\Lambda_{\alpha, \epsilon}^{\kappa, \lambda} = \bigcup_{\beta < \alpha} \Lambda_{\beta, \epsilon}^{\kappa, \lambda}$.
\item For $\alpha = \beta + 1$, we define $\Lambda_{\alpha, \epsilon}^{\kappa, \lambda}$ as the set of formulas of the form
\[
\varphi(\bar x) = \exists \bar y \bigwedge \{\psi(\bar x, \bar y) \mid \psi(\bar x, \bar y) \in \Psi\}
\]
for some $|\bar y| = \zeta < \lambda$\footnote{We allow $\zeta=0$.} and some $\Psi \subseteq \Lambda_{\beta, \epsilon + \zeta}^{\kappa, \lambda}$ with $|\Psi| < \kappa$.
\end{enumerate}
\end{definition}

\begin{remark}
We are primarily interested in the case when $\lambda=\aleph_0$.
\end{remark}

\begin{definition}\
\begin{enumerate}[(1)]
\item We define $\Lambda_{\alpha}^{\kappa, \lambda} =\bigcup_{\epsilon < \lambda} \Lambda_{\alpha, \epsilon}^{\kappa, \lambda}$.
\item We define $\Lambda_{\alpha}^{\infty, \lambda} = \bigcup \{ \Lambda_{\alpha}^{\kappa, \lambda} \, \mid \, \kappa \in \mrm{Card}\}$.
\item We define $\Lambda_{\infty}^{\kappa, \lambda} = \bigcup \{ \Lambda_{\alpha}^{\kappa, \lambda} \, \mid \, \alpha \in \mrm{Ord}\}$.

\item We define $\Lambda_{\infty}^{\infty, \infty} = \bigcup \{ \Lambda_{\alpha}^{\kappa, \lambda} \, \mid \, \alpha \in \mrm{Ord}, \kappa, \lambda \in \mrm{Card}\}$.
\end{enumerate}

\end{definition}

We say that $\varphi(\bar{x}),\psi(\bar{x})\in \mathfrak{L}_{\infty, \infty}(\tau_R)$ are \emph{equivalent} (or, more precisely, equivalent relative to the theory of $R$-modules) if, for every $R$-module $A$ and every $\bar a\in A^{|\bar x|}$, we have $A\models\varphi(\bar a)$ if and only if $A\models\psi(\bar a)$.

\begin{definition}
$\varphi(\bar{x})$ is a first-order positive primitive formula (a $\mrm{pp}$-formula, for short) if it is a first-order formula in the language of $R$-modules and is equivalent to a formula in $\Lambda_{\omega}^{\aleph_0, \aleph_0}$.
\end{definition}

We say that $A$ is a \emph{pure submodule} of $B$, denoted by $A\pleq B$, if $A\leq B$ and, for every first-order $\mrm{pp}$-formula $\varphi(\bar x)$ and every $\bar a\in A^{|\bar x|}$, we have $A\models\varphi(\bar a)$ if and only if $B\models\varphi(\bar a)$. Equivalently, every finite system of $R$-linear equations with right-hand side in $A$ that is solvable in $B$ is already solvable in $A$.

The following two results study the syntactic structure of infinitary $\mrm{pp}$-formulas.

\begin{proposition}\label{equiv-exist}
Let $\varphi(\bar x)\in\Lambda_{\alpha}^{\kappa,\aleph_0}$, where $\bar x=(x_0,\ldots,x_{n-1})$, $\kappa$ is a cardinal, and $\alpha$ is an ordinal. Then $\varphi(\bar x)$ is equivalent, relative to the theory of $R$-modules, to a formula of the form
\[
\exists\bar v\bigwedge_{i<\lambda}\delta_i(\bar v,\bar x)=0,
\]
with $\lambda\leq\kappa$\footnote{If $\kappa$ is regular, then $\lambda<\kappa$.}, $\bar v=(v_i)_{i<\lambda}$, and, for every $i<\lambda$,
\[
\delta_i(\bar v,\bar x)\doteq
\sum_{\ell<\lambda}r_{i,\ell}v_\ell+
\sum_{j<n}q_{i,j}x_j,
\]
with $r_{i,\ell},q_{i,j}\in R$ and for each $i < \lambda$, $r_{i,\ell}\neq0$ for only finitely many $\ell<\lambda$. In particular, $\varphi(\bar x)$ is equivalent to a formula in $\Lambda^{\lambda^+,\lambda^+}_1$.
\end{proposition}
\begin{proof}
The final assertion follows directly from the main statement. We prove by induction on $\alpha$ that the conclusion holds for every natural number $n$ and every $\varphi(\bar x)\in\Lambda_{\alpha,n}^{\kappa,\aleph_0}$.

\smallskip \noindent If $\alpha=0$ or $\alpha$ is a limit ordinal, there is nothing to prove. Assume that $\alpha=\beta+1$. Then
\[
\varphi(\bar x)=\exists\bar y\bigwedge\{\psi(\bar x,\bar y)\mid \psi(\bar x,\bar y)\in\Psi\}
\]
for some $|\bar y|=m<\aleph_0$ and some $\Psi\subseteq\Lambda_{\beta,n+m}^{\kappa,\aleph_0}$ with $|\Psi|<\kappa$. By the induction hypothesis, each $\psi(\bar x,\bar y)\in\Psi$ is equivalent to
\[
\exists\bar v \bigwedge_{i<\lambda_\psi}
\delta_{i,\psi}(\bar v,\bar x,\bar y)=0,
\]
where $\lambda_\psi\leq\kappa$ and
\[
\delta_{i,\psi}(\bar v,\bar x,\bar y)\doteq
\sum_{\ell<\lambda_\psi}r_{i,\ell,\psi}v_{\ell}
+\sum_{j<n}q_{i,j,\psi}x_j
+\sum_{j<m}s_{i,j,\psi}y_j,
\]
with $r_{i,\ell,\psi},q_{i,j,\psi},s_{i,j,\psi}\in R$ and for each $i < \lambda$, $r_{i,\ell,\psi}\neq0$ for only finitely many $\ell<\lambda_\psi$. After making the tuples $\bar v$ pairwise disjoint by tagging them with $\psi$ as $\bar{v}_\psi$ and setting $\lambda=\sum_{\psi\in\Psi}\lambda_\psi\leq\kappa$, we obtain the required equivalent formula
\[
\exists\bar y\exists(\bar v_\psi)_{\psi\in\Psi}
\bigwedge_{\psi\in\Psi}\bigwedge_{i<\lambda_\psi}
\delta_{i,\psi}(\bar v_\psi,\bar x,\bar y)=0.
\]
\end{proof}

\begin{proposition}\label{equiv-kappa}
Let $\varphi(\bar x)\in\Lambda_{\alpha}^{\kappa,\aleph_0}$, where $\bar x=(x_0,\ldots,x_{n-1})$, $\kappa$ is a cardinal, and $\alpha$ is an ordinal. If $\kappa$ is regular, then $\varphi(\bar x)$ is equivalent, relative to the theory of $R$-modules, to a formula in $\Lambda_{\kappa}^{\kappa,\aleph_0}$. If $\kappa$ is singular, then it is equivalent to a formula in $\Lambda_{\kappa^+}^{\kappa,\aleph_0}$.
\end{proposition}
\begin{proof}
We prove the regular case; the singular case is analogous. We show by induction on $\alpha\geq\kappa$ that every $\varphi(\bar x)\in\Lambda_{\alpha}^{\kappa,\aleph_0}$ is equivalent to a formula in $\Lambda_{\kappa}^{\kappa,\aleph_0}$.

\smallskip \noindent If $\alpha=\kappa$ or $\alpha$ is a limit ordinal, there is nothing to prove. Assume that $\alpha=\beta+1$. Then
\[
\varphi(\bar x)=\exists\bar y\bigwedge\{\psi(\bar x,\bar y)\mid \psi(\bar x,\bar y)\in\Psi\}
\]
for some $|\bar y|=m<\aleph_0$ and some $\Psi\subseteq\Lambda_{\beta}^{\kappa,\aleph_0}$ with $|\Psi|<\kappa$. By the induction hypothesis, each $\psi(\bar x,\bar y)\in\Psi$ is equivalent to a formula $\theta_\psi(\bar x,\bar y)\in\Lambda_{\kappa}^{\kappa,\aleph_0}$. Choose $\alpha_\psi<\kappa$ such that $\theta_\psi(\bar x,\bar y)\in\Lambda_{\alpha_\psi}^{\kappa,\aleph_0}$, and let $\gamma=\sup_{\psi\in\Psi}\alpha_\psi<\kappa$, using the regularity of $\kappa$. Then
\[
\exists\bar y\bigwedge\{\theta_\psi(\bar x,\bar y)\mid\psi\in\Psi\}
\in\Lambda_{\gamma+1}^{\kappa,\aleph_0}\subseteq\Lambda_{\kappa}^{\kappa,\aleph_0},
\]
and this formula is equivalent to $\varphi(\bar x)$.
\end{proof}

  \begin{remark}\label{equiv-remark}
 The arguments of Propositions~\ref{equiv-exist} and~\ref{equiv-kappa} show that every formula in $\Lambda_{\infty}^{\aleph_0,\aleph_0}$ is equivalent to a formula in $\Lambda_{1}^{\aleph_0,\aleph_0}$.
 \end{remark}

The following result will be very useful.

\begin{proposition}\label{preserve}
Let $\lambda\leq\kappa$ be cardinals, let $\alpha$ and $\epsilon$ be ordinals, and let $\varphi(\bar x)\in\Lambda_{\alpha,\epsilon}^{\kappa,\lambda}$.
\begin{enumerate}[(1)]
\item If $A\models\varphi(\bar a)$ for $\bar a\in A^\epsilon$ and $f:A\to B$ is an $R$-homomorphism, then $B\models\varphi(f(\bar a))$.
\item The set $\varphi(A)=\{\bar a\in A^\epsilon\mid A\models\varphi(\bar a)\}$ is a subgroup of $A^\epsilon$.
\end{enumerate}
\end{proposition}
\begin{proof}
We prove (1); the proof of (2) is analogous. Let $f:A\to B$ be an $R$-homomorphism. We show by induction on $\alpha$ that, for every ordinal $\epsilon<\lambda$ and every $\varphi(\bar x)\in\Lambda_{\alpha,\epsilon}^{\kappa,\lambda}$, if $A\models\varphi(\bar a)$ for $\bar a\in A^\epsilon$, then $B\models\varphi(f(\bar a))$.

\smallskip \noindent If $\alpha=0$, the result is clear because $\varphi(\bar x)$ is a linear equation; if $\alpha$ is a limit ordinal, it follows immediately from the induction hypothesis. Assume that $\alpha=\beta+1$ and
\[
\varphi(\bar x)=\exists\bar y\bigwedge\{\psi(\bar x,\bar y)\mid \psi(\bar x,\bar y)\in\Psi\},
\]
for some $|\bar y|=\zeta<\lambda$ and some $\Psi\subseteq\Lambda_{\beta,\epsilon+\zeta}^{\kappa,\lambda}$ with $| \Psi| < \kappa$. Choose $\bar c\in A^\zeta$ such that
$A\models\bigwedge\{\psi(\bar a,\bar c)\mid\psi(\bar x,\bar y)\in\Psi\}.$
By the induction hypothesis, $B\models\psi(f(\bar a),f(\bar c))$ for every $\psi\in\Psi$. Hence $B\models\varphi(f(\bar a))$.
\end{proof}

\begin{definition}
Given $\Sigma \subseteq \Lambda_{\infty}^{\infty, \infty}$, let $\leq_\Sigma$ be the binary relation given in $\rmod$ by $A \leq_\Sigma B$ if and only if $A \leq B$ and for every $\varphi(\bar{x}) \in \Sigma$ and every $\bar{a} \in A^{|\bar{x}|}$, $A \models \varphi(\bar{a})$ if and only if $B \models \varphi(\bar{a})$.
\end{definition}

\begin{corollary}\label{ref_sums} Assume $\Sigma \subseteq \Lambda_{\infty}^{\infty, \infty}$. If $A \leq_{\oplus} B$, then $A \leq_{\Sigma} B$.
\end{corollary}
\begin{proof}
Since $A\leq_{\oplus}B$, there is an $R$-homomorphism $\pi:B\to A$ such that $\pi\restriction A=\mrm{id}_A$. The result now follows from Proposition~\ref{preserve}.
\end{proof}

The following classical construction from module theory will be useful.

\begin{remark}\label{pushout}
The pushout of a pair of inclusions $A\leq B$ and $A\leq C$ in the category of $R$-modules is a triple
\[
(B\oplus_A C=(B\oplus C)/\widehat A,\; f:B\to B\oplus_A C, \; g:C\to B\oplus_A C)
.
\]
with $\widehat A=\{(a,-a)\mid a\in A\}$, $f: b \mapsto [(b,0)]$ and $g: c \mapsto [(0,c)]$.

Observe that $f,g$ are monomorphisms. Moreover, whenever $f':B\to D$ and $g':C\to D$ are $R$-homomorphisms satisfying $f'\restriction A=g'\restriction A$, the unique $R$-homomorphism $t:B\oplus_A C\to D$ such that $t\circ f=f'$ and $t\circ g=g'$ is given by $
t([(b,c)])=f'(b)+g'(c)$.


\end{remark}

\subsection{Abstract elementary classes}

Abstract elementary classes (AECs) were introduced by Shelah in the seventies \cite{sh88}. We quickly introduce the notions of AECs used in this paper. For more details,  see \cite[\S\S 4--8]{baldwinbook09}.

The following definition is due to Grossberg.
\begin{definition}\label{def_AC}
Let $\mathbf{K}$ be a class of $L$-structures and let $\preccurlyeq$ be a binary relation on $\mathbf{K}$. We say that $(\mathbf{K}, \preccurlyeq)$ is an {\em abstract class} if the following conditions are satisfied:
\begin{enumerate}[(1)]
	\item $\mathbf{K}$ and $\preccurlyeq$ are closed under isomorphisms, i.e., if $A\in \mathbf{K}$ and $f:A\to B$ is an isomorphism then $B\in \mathbf{K}$ and if $C\in \mathbf{K}$ is such that $C\preccurlyeq A$ then $f(C)\preccurlyeq B$;
	\item if $A \preccurlyeq B$, then $A$ is an $L$-submodel of $B$ (written $A \leq B$);
	\item the relation $\preccurlyeq$ is a partial order on $\mathbf{K}$.
\end{enumerate}
\end{definition}

\begin{definition}
Let $(\mathbf{K},\preccurlyeq)$ be an abstract class, and let $A,B\in\mathbf{K}$. An embedding $f:A\to B$ is a \emph{strong embedding} (or \emph{$\prec$-embedding}) if $f(A)\preccurlyeq B$.
\end{definition}

\begin{definition}\label{def_AEC}
Let $\mathbf K$ be a class of $L$-structures and let $\preccurlyeq$ be a binary relation on $\mathbf K$. We say that $(\mathbf K,\preccurlyeq)$ is an \emph{abstract elementary class} (AEC) if it is an abstract class, that is, if it satisfies Definition~\ref{def_AC}(1)--(3), and the following additional conditions hold:
	\begin{enumerate}[(1)]\setcounter{enumi}{3}
		\item \emph{Tarski--Vaught axioms}: If $(A_i)_{i < \delta}$ is an increasing continuous $\preccurlyeq$-chain of structures in $\mathbf{K}$ for $\delta$ a limit ordinal, then:
		\begin{enumerate}[(4.1)]
			\item $\bigcup_{i < \delta} A_i \in \mathbf{K}$;
			\item for each $j < \delta$, $A_j \preccurlyeq \bigcup_{i < \delta} A_i$;
			\item \emph{Smoothness axiom}: if $A_i \preccurlyeq B$ for all $i<\delta$, then $\bigcup_{i < \delta} A_i \preccurlyeq B$.
	\end{enumerate}
			\item \emph{Coherence axiom}: If $A, B, C \in \mathbf{K}$, $A \preccurlyeq C$, $B \preccurlyeq C$ and $A \leq B$, then $A \preccurlyeq B$;
			\item \emph{L\"owenheim--Skolem--Tarski axiom}: There is a cardinal $\lambda\geq |L|+\aleph_0$ such that, whenever $B\in\mathbf K$ and $X\subseteq B$, there is $A\in\mathbf K$ with $X\subseteq A\preccurlyeq B$ and $|A|\leq |X|+\lambda$. The \emph{L\"owenheim--Skolem number} of $\mathcal K$, denoted by $\mathrm{LS}(\mathcal K)$, is the least such cardinal $\lambda\geq |L|+\aleph_0$.
\end{enumerate}
When $(\mathbf{K},\preccurlyeq)$ is an AEC, we refer to $\preccurlyeq$ as a \emph{strong submodel relation} on the class $\mathbf{K}$.
\end{definition}

\begin{remark}
In this paper, $L$ is always the language of $R$-modules, and $\mathbf K$ is usually the class of all $R$-modules.
\end{remark}

Throughout the rest of this subsection, $\mathcal K=(\mathbf K,\preccurlyeq)$ is an AEC.

\begin{definition}\label{amal} $\mathcal K$ has 
\begin{enumerate}[$\bullet$]
\item the \emph{amalgamation property} if every span $A\preccurlyeq B,C$ in $\mathcal{K}$ can be completed to a commutative square of strong embeddings;
\item  the \emph{joint embedding property} if any two models in $\mathbf{K}$ strongly embed into a third model in $\mathbf{K}$;
\item \emph{no maximal models} if every model in $\mathbf{K}$ has a proper strong extension in $\mathbf{K}$. 
\end{enumerate}
\end{definition}

For $\lambda\geq\mathrm{LS}(\mathcal K)$, let $
\mathbf K_\lambda=\{A\in\mathbf K\mid |A|=\lambda\}$.

\begin{definition}\label{notation_types}
Let $A\in\mathbf K$,  $X\subseteq A$, and $\bar b$ be a sequence from $A$. The (orbital or Galois) \emph{type of $\bar b$ over $X$ in $A$}, denoted by $\gtp_{\mathcal K}(\bar b/X;A)$, is the equivalence class of $(\bar b,X,A)$ modulo $E_{\mathcal K}$, where $E_{\mathcal K}$ is the transitive closure of the relation $E_{\mathcal K}^{\mrm{at}}$ defined as follows.

 For $\ell\in\{1,2\}$,  $A_\ell\in\mathbf K$,  $\bar b_\ell \in A_\ell^{<\infty}$, and $X_\ell\subseteq A_\ell$, we set
\[
(\bar b_1,X_1,A_1)\mathrel{E_{\mathcal K}^{\mrm{at}}}(\bar b_2,X_2,A_2)
\]
if $X:=X_1=X_2$ and there are $A_*\in\mathbf K$ and strong embeddings $f_\ell:A_\ell\to A_*$, for $\ell\in\{1,2\}$, such that
$f_1\restriction X=f_2\restriction X=\mrm{id}_X$ and  $f_1(\bar b_1)=f_2(\bar b_2)$.


\end{definition}

Stability is a key model-theoretic dividing line that we study in this paper.

\begin{definition}\label{def_stability}\
\begin{enumerate}[(1)]
\item If $A\in\mathbf K$, let
\[
\gS_{\mathcal K}(A)=\{\gtp_{\mathcal K}(b/A;B)\mid A\preccurlyeq B\in\mathbf K\text{ and }b\in B\}.
\]
\item Let $\lambda\geq\mathrm{LS}(\mathcal K)$. We say that $\mathcal K$ is $\lambda$-stable if $|\gS_{\mathcal K}(A)|\leq\lambda$ for every $A\in\mathbf K_\lambda$.

\item $\mathcal K$ is stable if it is $\mu$-stable for unboundedly many cardinals $\mu$.
\end{enumerate}

\end{definition}

Given $p=\gtp_{\mathcal K}(b/A;B)\in\gS_{\mathcal K}(A)$ and $X\subseteq A$, let $p\restriction X=[(b,X,B)]_{E_{\mathcal K}}$. The following notion was isolated in \cite{tamenessone}.

\begin{definition} $\mathcal{K}$ is $(<\aleph_0)$-tame if, for every $A \in \mathbf{K}$ and $p, q \in \gS_{\mathcal K}(A)$, if $ p \neq q$, then there is $X$ a finite subset of $A$ such that $p \restriction X \neq q \restriction X$.
\end{definition}

\section{The lattice structure of $\mathscr{L}_{R}$}

The main object of this paper is the following lattice.

\begin{definition}\label{main_def}
Let $R$ be a ring and let $\lambda\geq\operatorname{card}(R)+\aleph_0$ be a cardinal. Define $\mathscr{L}_{R,\lambda}$ to be the partially ordered class\footnote{Corollary~\ref{lattice} shows that it is a lattice.} of all binary relations $\preccurlyeq$ on $\rmod$ such that $(\rmod,\preccurlyeq)$ is an AEC, $\mathrm{LS}(\rmod,\preccurlyeq)\leq\lambda$, and $\preccurlyeq$ refines the direct-summand relation. We order $\mathscr{L}_{R,\lambda}$ by reverse inclusion.\footnote{Thus stronger strong submodel relations lie higher in the lattice.} Let
\[
\mathscr{L}_{R}=\bigcup_{\lambda\geq\operatorname{card}(R)+\aleph_0}\mathscr{L}_{R,\lambda}.
\]

\end{definition}

In this section, we establish basic properties of the lattices $\mathscr{L}_{R}$ and $\mathscr{L}_{R,\lambda}$.

\subsection{Basic results}\label{b-results}
Before we proceed it is important to notice that the structure of $\mathscr{L}_{R}$ and $\mathscr{L}_{R, \lambda}$ depends strongly on the ring $R$. For instance, if $R$ is left semisimple, then $\mathscr{L}_{R}$ and $\mathscr{L}_{R, \lambda}$ consists only of the submodule relation.

\begin{proposition}\label{join}
Let $\lambda$ be a cardinal.
\begin{enumerate}[(1)]
\item If $\{\preccurlyeq_i\mid i<\lambda\}\subseteq\mathscr{L}_{R,\lambda}$, then $\bigcap_{i<\lambda}\preccurlyeq_i\in\mathscr{L}_{R,\lambda}$.
\item If $\{\preccurlyeq_i\mid i<\theta\}\subseteq\mathscr{L}_{R}$ for some cardinal $\theta > 0$, then $\bigcap_{i<\theta}\preccurlyeq_i\in\mathscr{L}_{R}$.
\end{enumerate}
\end{proposition}

\begin{proof}
We first prove (1). It is straightforward to verify all the AEC axioms for $(\rmod,\bigcap_{i<\lambda}\preccurlyeq_i)$ except the L\"owenheim--Skolem--Tarski axiom. Let $X\subseteq B\in\rmod$, and set $\mu=|X|+\lambda$.

\smallskip \noindent We recursively construct $\{A_{i,n}\mid i<\lambda,\ n<\omega\}$ inside $B$, starting with $X\subseteq A_{0,0}$, so that:
\begin{enumerate}[(1)]
\item $|A_{i,n}|\leq\mu$ and $A_{i, n} \preccurlyeq_i 	B$ for every $i<\lambda$ and $n<\omega$;
\item $(A_{i,n})_{n<\omega}$ is an increasing $\preccurlyeq_i$-chain for every $i<\lambda$;
\item $A_{i,n}\subseteq A_{j,n}$ whenever $i<j<\lambda$ and $n<\omega$;
\item $\bigcup_{i<\lambda}A_{i,n}\subseteq A_{0,n+1}$ for every $n<\omega$.
\end{enumerate}

\smallskip \noindent The construction follows by recursion on $n$, using the L\"owenheim--Skolem--Tarski axiom in each $(\rmod,\preccurlyeq_i)$ inside $B$ and then the Coherence axiom.

\smallskip \noindent Let $A=\bigcup_{n<\omega}A_{0,n}$. Then $X\subseteq A$ and $|A|\leq\mu$ by Condition~(1). Fix $i<\lambda$. Conditions~(3) and~(4) give $A=\bigcup_{n<\omega}A_{i,n}$, so Conditions~(1) and (2) and the Tarski--Vaught axioms imply that $A\preccurlyeq_i B$. Thus $A\,(\bigcap_{i<\lambda}\preccurlyeq_i)\,B$.

\smallskip \noindent We now prove (2). The same construction works after replacing $\mu$ in Condition~(1) by the size bound 
$\mu=|X|+\theta+\sup_{i<\theta}\mathrm{LS}(\rmod,\preccurlyeq_i).$
Consequently, we have
\[
\mathrm{LS}\left(\rmod,\bigcap_{i<\theta}\preccurlyeq_i\right)
\leq \theta+\sup_{i<\theta}\mathrm{LS}(\rmod,\preccurlyeq_i),
\]
and hence $\bigcap_{i<\theta}\preccurlyeq_i\in\mathscr{L}_{R}$.
\end{proof}

\begin{proposition}\label{meet}
Let $\lambda$ be a cardinal and $0<\theta$ be a cardinal (possibly finite).
\begin{enumerate}[(1)]
\item If $\{\preccurlyeq_i\mid i<\theta\}\subseteq\mathscr{L}_{R,\lambda}$, then this family has a greatest lower bound in $\mathscr{L}_{R,\lambda}$.
\item If $\{\preccurlyeq_i\mid i<\theta\}\subseteq\mathscr{L}_{R}$, then this family has a greatest lower bound in $\mathscr{L}_{R}$.
\end{enumerate}
\end{proposition}

\begin{proof}
We prove (1); the proof of (2) is identical. Let
\[
\Delta=\{\preccurlyeq\in\mathscr{L}_{R,\lambda}\mid \,
\preccurlyeq_i\subseteq\preccurlyeq\text{ for every }i<\theta\} \text{ and }\preccurlyeq^{*}=\bigcap\Delta.
\]

\smallskip \noindent The set $\Delta$ is nonempty because the submodule relation belongs to it. It is straightforward to verify all the AEC axioms for $(\rmod,\preccurlyeq^{*})$ except the L\"owenheim--Skolem--Tarski axiom. Let $X\subseteq B\in\rmod$. Applying the L\"owenheim--Skolem--Tarski axiom to $X$ in $(\rmod,\preccurlyeq_0)$ inside $B$, there is $A\preccurlyeq_0 B$ such that $X\subseteq A$ and $|A|\leq|X|+\lambda$. Since $\preccurlyeq_0\subseteq\preccurlyeq$ for every $\preccurlyeq\in\Delta$, we have $A\preccurlyeq^{*}B$. Hence $\preccurlyeq^{*}\in\mathscr{L}_{R,\lambda}$, and it is straightforward to show that it is the greatest lower bound of the family.

\smallskip \noindent For (2), define $\Delta$ inside $\mathscr{L}_{R}$ and use $\mathrm{LS}(\rmod,\preccurlyeq_0)$ in place of $\lambda$.
\end{proof}

\begin{remark}
In general, the relation $\preccurlyeq^{*}$ from Proposition~\ref{meet} differs from $\bigcup_{i<\theta}\preccurlyeq_i$. For example, let $R=\mathbb Z$. Define $A\preccurlyeq_1B$ by requiring that divisibility by $2$ of elements of $A$ be preserved between $A$ and $B$, and define $A\preccurlyeq_2B$ analogously using divisibility by $3$. $\preccurlyeq_1 \cup \preccurlyeq_2 \neq \preccurlyeq^*$ because $(\rmod, \preccurlyeq_1 \cup \preccurlyeq_2)$ fails to be an AEC since transitivity fails as witnessed by $
6\mathbb Z\leq 3\mathbb Z\leq\mathbb Z$.
\end{remark}

\begin{corollary}\label{lattice}
For every ring $R$ and every cardinal $\lambda\geq\operatorname{card}(R)+\aleph_0$, $\mathscr{L}_{R}$ and $\mathscr{L}_{R,\lambda}$ are lattices whose bottom element is the submodule relation.
\end{corollary}
\begin{proof}
Proposition~\ref{join} provides joins (least upper bound), and Proposition~\ref{meet} provides meets (greatest lower bound).
\end{proof}

The following result relies heavily on \cite[Fact 3.1.1]{babo}.

\begin{lemma}\label{bound-size}
For every ring $R$ and every cardinal $\lambda\geq\operatorname{card}(R)+\aleph_0$, $
|\mathscr{L}_{R,\lambda}|\leq 2^{2^\lambda}$.
\end{lemma}
\begin{proof} It is enough to show that $\mathscr{L}_{R, \lambda}^{=}:= \{ \preccurlyeq \in \mathscr{L}_{R, \lambda} \mid  \mrm{LS}(\rmod, \preccurlyeq) = \lambda \}$ is such that $| \mathscr{L}_{R, \lambda}^{=} | \leq 2^{2^\lambda}$ as $\mathscr{L}_{R, \lambda} = \bigcup_{\mu \leq \lambda} \mathscr{L}_{R, \mu}^{=}$.

\smallskip \noindent Let $L_1$ be the expansion of $\tau_R$ by
$ \{ F_i^n(\bar{x}) \mid i < \lambda, n <\omega \text{ and } |\bar{x}| = n\}$.
For each $\preccurlyeq\in\mathscr{L}_{R,\lambda}^{=}$, the proof of Shelah's Presentation Theorem yields a complete first-order $L_1$-theory $T_{\preccurlyeq}$ and a set $\Gamma_{\preccurlyeq}$ of $L_1$-types satisfying Conditions~(1)--(3) of \cite[Fact 3.1.1]{babo}.

\smallskip \noindent Let $\Phi: \mathscr{L}_{R, \lambda}^{=} \to \{ T \mid T \text{ is a } L_1\text{-theory} \} \times \{ \Gamma \mid  \Gamma \text{ is a set of } L_1\text{-types}\}$ given by $\Phi(\preccurlyeq)= (T_{\preccurlyeq}, \Gamma_{\preccurlyeq})$.
There are at most $2^\lambda$ possible $L_1$-theories and at most $2^{2^\lambda}$ possible sets of $L_1$-types, so it suffices to prove that $\Phi$ is injective.

\smallskip \noindent Suppose
$\Phi(\preccurlyeq)=\Phi(\preccurlyeq^{*})$ and $A\preccurlyeq B$. By Condition~(1) of \cite[Fact 3.1.1]{babo}, choose $A'\in\operatorname{EC}(T_{\preccurlyeq},\Gamma_{\preccurlyeq})$ with $A'\restriction\tau_R=A$. By Condition~(3) of \cite[Fact 3.1.1]{babo}, extend $A'$ to $B'\in\operatorname{EC}(T_{\preccurlyeq},\Gamma_{\preccurlyeq})$ such that $A'\leq B'$ as $L_1$-structures and $B'\restriction\tau_R=B$. The equality of the two presentation pairs and another application of Condition~(3) of \cite[Fact 3.1.1]{babo}, now for $\preccurlyeq^{*}$, give $A\preccurlyeq^{*}B$. Thus $\preccurlyeq\subseteq\preccurlyeq^{*}$; the reverse inclusion can be shown analogously.
\end{proof}

\begin{proposition}
If $\preccurlyeq \, \in \, \mathscr{L}_{R}$, then $(\rmod, \preccurlyeq)$ has no maximal models and the joint embedding property.
\end{proposition}
\begin{proof}
This follows by taking direct sums, since every direct-summand embedding is a $\preccurlyeq$-embedding.
\end{proof}

\subsection{When is $\oleq\in\mathscr{L}_{R}$?}

	We now characterize the rings $R$ for which $\oleq\in\mathscr{L}_{R}$, equivalently, those for which $(\rmod,\oleq)$ is an AEC. The following proposition is immediate, so we omit its proof.

\begin{proposition}\label{1.2}
Let $B$ be an $R$-module, let $\delta$ be a limit ordinal, and let $(A_\alpha)_{\alpha<\delta}$ be a family of submodules of $B$. If $\sum_{\alpha<\beta}A_\alpha$ is direct for every $\beta<\delta$, then $\sum_{\alpha<\delta}A_\alpha$ is direct.
\end{proposition}


\begin{lemma}\label{1.3}
Assume that $(\rmod,\leq_\oplus)$ is an AEC with $\mathrm{LS}(\rmod,\leq_\oplus)=\kappa$. Then every $R$-module $B$ admits a decomposition
$B=\bigoplus_{A\in\mathcal X}A$
for some family $\mathcal X$ of submodules of $B$, each of cardinality at most $\kappa$.
\end{lemma}

\begin{proof}
Let $\mu=|B|$, and fix an enumeration $(b_\alpha)_{\alpha<\mu}$ of $B$.

\smallskip \noindent We recursively construct a sequence $(A_\alpha)_{\alpha<\mu}$ of submodules of $B$, each of cardinality at most $\kappa$, such that for every $\alpha<\mu$:
\begin{enumerate}[(1)]
\item the sum $\sum_{\gamma<\alpha+1}A_\gamma$ is direct;
\item $\bigoplus_{\gamma<\alpha+1}A_\gamma$ is a direct summand of $B$;
\item $b_\alpha\in\bigoplus_{\gamma<\alpha+1}A_\gamma$.
\end{enumerate}

\smallskip \noindent For $\alpha=0$, apply the L\"owenheim--Skolem--Tarski axiom to $b_0$ inside $B$ to obtain a direct summand $A_0$ of cardinality at most $\kappa$ containing $b_0$.

\smallskip \noindent Suppose $\alpha=\delta+1$. By induction hypothesis, we have that 
$B=C\oplus D$ with $C=\bigoplus_{\gamma<\delta+1}A_\gamma.$
Let $\pi:B\to D$ be the projection given by the decomposition. Applying the L\"owenheim--Skolem--Tarski axiom inside $D$ to $\pi(b_\alpha)$, choose a direct summand $A_\alpha$ of $D$ containing $\pi(b_\alpha)$ and satisfying $|A_\alpha|\leq\kappa$. Then $(A_\gamma)_{\gamma<\alpha+1}$ has the required properties.

\smallskip \noindent Suppose that $\alpha$ is a limit ordinal. By induction hypothesis, for any $\delta<\alpha$, the sum $\sum_{\gamma<\delta}A_\gamma$ is direct and $\bigoplus_{\gamma<\delta } A_{\gamma}$ is a direct summand of $B$. Now, by Proposition \ref{1.2} we have that the sum $\sum_{\gamma<\alpha}A_\gamma$ is direct, and moreover: 
\[
\bigoplus_{\gamma<\alpha}A_\gamma
=
\bigcup_{\delta<\alpha}\bigoplus_{\gamma<\delta}A_\gamma.
\]
Thus, by smoothness, this union is a direct summand of $B$. We may therefore proceed exactly as in the successor case to choose $A_\alpha$.

\smallskip \noindent Finally, Conditions~(1) and~(3), together with Proposition~\ref{1.2}, yield
$B=\bigoplus_{\alpha<\mu}A_\alpha$. Taking $\mathcal X=\{A_\alpha\mid\alpha<\mu\}$ completes the proof.
\end{proof}

\begin{fact}[{\cite[Theorem~4.5.7]{prest_book2}}]\label{1.4}
  Let $R$ be a ring. The following are equivalent:
  \begin{enumerate}[(1)]
  \item $R$ is left pure semisimple, i.e., $\leq_{\oplus}=\pleq$;
  \item There is a cardinal $\kappa$ such that every left $R$-module is a direct sum of modules of cardinality less than $\kappa$.
\end{enumerate}
\end{fact}

		\begin{theorem}\label{when_direct_summand}
Let $R$ be a ring. $(\rmod,\oleq)$ is an AEC if and only if $R$ is left pure semisimple.
\end{theorem}

\begin{proof}
  The forward implication follows from Lemma~\ref{1.3} and Fact~\ref{1.4}. Conversely, if $R$ is left pure semisimple, then $\oleq=\pleq$, and $(\rmod,\pleq)$ is an AEC (see for example Proposition~\ref{AEC-plea}).
\end{proof}

If $\oleq\in\mathscr{L}_{R}$, then $\mathscr{L}_{R}$ has a top element. This leaves open the following natural question.

 \begin{question}\label{top-question}
 Does the lattice $\mathscr{L}_{R}$ have a top element for every ring $R$? If the answer is negative, characterize the rings for which $\mathscr{L}_{R}$ has a top element.
 \end{question}

\section{The sublattices $\mathscr{L}^{1}_{R}$ and $\mathscr{L}^{2}_{R}$}

We study the lattice  $\mathscr{L}_{R}$ by dividing it into two natural sublattices: the one below purity and the one above purity.

\subsection{The sublattice $\mathscr{L}^1_{R}$}

We first study the sublattice given by the strong submodel relations  on $\rmod$ which are weaker than the pure submodule relation.

\begin{definition}
Given a ring $R$, let $\mathscr{L}^{1}_{R}$ be the sublattice of $\mathscr{L}_{R}$ below the pure submodule relation.
\end{definition}

The structure of $\mathscr{L}^{1}_{R}$ again depends strongly on $R$. For instance, if $R$ is von Neumann regular, then purity coincides with the submodule relation \cite[2.3.22]{prest_book2}, so $\mathscr{L}^{1}_{R}$ has a single element. In this section, we show that, in general, $\mathscr{L}^{1}_{R}$ contains many well-behaved relations. In  Section~\ref{5.1}, we further investigate $\mathscr{L}^{1}_{\mathbb{Z}}$.

\begin{remark}\label{easy-rem}
If $\preccurlyeq\in\mathscr{L}_{R}$ lies below $\pleq$, then
$\mathrm{LS}(\rmod,\preccurlyeq)=\operatorname{card}(R)+\aleph_0$, because
$\mathrm{LS}(\rmod,\pleq)=\operatorname{card}(R)+\aleph_0$. Thus, when studying the sublattice below $\pleq$, restricting from $\mathscr{L}_{R}$ to $\mathscr{L}_{R, \lambda}$ for any cardinal $\lambda$ does not change the class of relations.
\end{remark}

\begin{proposition}\label{basic_L1}
 $\mathscr{L}^{1}_{R}$ is a lattice with bottom element given by $\leq$ and top element given by $\pleq$. Moreover, $| \mathscr{L}^{1}_{R}| \leq 2^{2^{\operatorname{card}(R) + \aleph_0}}$. 
\end{proposition}
\begin{proof}
This follows from Corollary \ref{lattice}. The \emph{moreover part} follows from Remark \ref{easy-rem} and Lemma \ref{bound-size}.
\end{proof}

\begin{proposition}\label{AEC-plea}
If $\Sigma$ is a set of first-order $\mrm{pp}$-formulas, then $\leq_\Sigma\in\mathscr{L}^{1}_{R}$.
\end{proposition}
\begin{proof}
This follows directly from the syntactic form of $\mrm{pp}$-formulas and the first-order downward L\"owenheim--Skolem theorem.
\end{proof}

\begin{example}\label{ex-1}
The following binary relations are in $\mathscr{L}^{1}_{R}$:
\begin{enumerate}[(1)]
\item $A$ is a submodule of $B$.
\item $A$ is a pure submodule of $B$.
\item  $A$ is an $\text{RD}$-submodule of $B$, i.e,   $A \leq B$ and for every $r\in R$, $rB \cap A = rA$. We denote it by $A \leq_\mrm{RD} B$.

\end{enumerate}
	
\end{example}

	\begin{definition}\label{def_pleasant}
A set $\Sigma$ of first-order $\mrm{pp}$-formulas is \emph{pleasant} if:
\begin{enumerate}[(1)]
\item every $\varphi(\bar x)\in\Sigma$ has the form
\[
\exists\bar y\,(H\bar y=K\bar x),
\]
with $H\in\operatorname{Mat}_{m\times n}(R)$, $\bar y=(y_1,\ldots,y_n)$, $K\in\operatorname{Mat}_{m\times s}(R)$, and $\bar x=(x_1,\ldots,x_s)$;
\item if $\exists\bar y\,(H\bar y=K\bar x)\in\Sigma$, then
$\exists\bar y\,(H\bar y=\bar z)\in\Sigma$ for $\bar z=(z_1,\ldots,z_m)$.
\end{enumerate}
A binary relation $\preccurlyeq$ on $\rmod$ is \emph{pleasant with respect to $\Sigma$} if $\Sigma$ is pleasant and $\preccurlyeq=\leq_\Sigma$. We say  $\preccurlyeq$ is \emph{pleasant} if it is pleasant with respect to some such $\Sigma$.
\end{definition}

	\begin{example}\label{points-(I)}\
\begin{enumerate}[(1)]
\item The submodule relation is pleasant with respect to $\Sigma=\emptyset$.
\item The RD-submodule relation is pleasant with respect to
\[
\Sigma=\{\exists y\,(ry=x)\mid r\in R\}.
\]
\item The pure submodule relation is pleasant with respect to
\[
\Sigma=\{\exists\bar y\,(H\bar y=K\bar x)\mid
H\in\operatorname{Mat}_{m\times n}(R),\ K\in\operatorname{Mat}_{m\times s}(R),\ m,n,s\in\mathbb N\}.
\]
\end{enumerate}
\end{example}

	\begin{remark}
		Every pleasant relation belongs to $\mathscr{L}^{1}_{R}$ by Proposition~\ref{AEC-plea}.
	\end{remark}

	\begin{lemma}\label{AP-(I)}
If $\preccurlyeq$ is pleasant, then $(\rmod,\preccurlyeq)$ has the amalgamation property.
\end{lemma}
\begin{proof}
Let $A\preccurlyeq B,C$, and form the pushout $(B\oplus_A C, f:B\to B\oplus_A C, g:C\to B\oplus_A C)$ with maps as in Remark~\ref{pushout}.

\smallskip \noindent Choose a pleasant set $\Sigma$ such that $\preccurlyeq=\leq_\Sigma$. The maps $f$ and $g$ are monomorphisms, so it remains to show that the formulas in $\Sigma$ are reflected. We verify this for $f$; the proof for $g$ is analogous.

\smallskip \noindent Suppose
$B\oplus_A C\models\exists\bar y\,(H\bar y=Kf(\bar\ell)),$
where $\exists\bar y\,(H\bar y=K\bar x)\in\Sigma$ and $\bar\ell\in B^s$. Choose $[(\bar b,\bar c)]\in(B\oplus_A C)^n$ witnessing this formula. Then there is $\bar a\in A^m$ such that
\[
B\models H\bar b-K\bar\ell=\bar a,
\qquad
C\models H\bar c=-\bar a.
\]
As $\Sigma$ is pleasant, $\exists\bar y\,(H\bar y=\bar z)\in\Sigma$. Since $A\leq_\Sigma C$, there is $\bar e\in A^n$ such that $A\models H\bar e=-\bar a$. So $B \models  H\bar e=-\bar a$. Adding the last two equations in $B$ we have that $B \models H(\bar b+\bar e)=K(\bar\ell)$. Therefore,
$B\oplus_A C\models Hf(\bar b+\bar e)=Kf(\bar\ell)$ by Proposition \ref{preserve}.
Hence $f[B]\leq_\Sigma B\oplus_A C$.
\end{proof}
	
	\begin{definition}
Let $\Sigma$ be a set of $\mrm{pp}$-formulas, let $b\in B$, and let $U\subseteq B$. Define $\operatorname{pp}_{\Sigma}(b/U;B)$ to be
\[
\begin{aligned}
&\{\varphi(z;\bar u)\mid \varphi(z,\bar x)\in\Sigma,\ \bar u\in U^{|\bar x|},\ B\models\varphi(b,\bar u)\}\\
&\quad\cup\{sz+ru=0\mid s,r\in R,\ u\in U,\ B\models sb+ru=0\}\\
&\quad\cup\bigl\{\exists\bar y\,(H\bar y=\bar u+\bar s^{T}z)\,\bigm|\,
\exists\bar y\,(H\bar y=K\bar x)\in\Sigma,\ 
\bar u\in U^{<\omega},\ \bar s\in R^{<\omega},\\
&\hspace{47mm} B\models\exists\bar y\,(H\bar y=\bar u+\bar s^{T}b)\bigr\}.
\end{aligned}
\]
If $\Sigma$ is the set of all $\mrm{pp}$-formulas, we write $\operatorname{pp}(b/U;B)$.
\end{definition}

\begin{remark} The formulas on the second line and third line of the previous definition are themselves first-order pp-formulas. 
	Therefore, when $\Sigma$ is the set of all pp-formulas, the last two lines do not add any additional information.
	\end{remark}
	
	
\begin{theorem}\label{char-gtps}
Assume that $\preccurlyeq$ is pleasant with respect to $\Sigma$, and let $A\leq B,C$. The following are equivalent:
\begin{enumerate}[(1)]
\item $\gtp_{(\rmod,\preccurlyeq)}(m/A;B)=\gtp_{(\rmod,\preccurlyeq)}(n/A;C)$;
\item $\operatorname{pp}_{\Sigma}(m/A;B)=\operatorname{pp}_{\Sigma}(n/A;C)$.
\end{enumerate}
\end{theorem}
\begin{proof}
Suppose that $\gtp_{(\rmod,\preccurlyeq)}(m/A;B)=\gtp_{(\rmod,\preccurlyeq)}(n/A;C)$. By Lemma~\ref{AP-(I)}, equality of  types may be witnessed in a common extension, so there are $D \in \rmod$ and $f: B \to D$ a $\preccurlyeq$-embedding such that $C \preccurlyeq D$, $f\upharpoonright A = \mathrm{id}_A$ and $f(m)=n$. Since $\preccurlyeq=\leq_\Sigma$, the formulas from the first line in the definition of $\operatorname{pp}_\Sigma$ agree for $m$ and $n$. Since $f$ is a monomorphism, the same is true of the equations in the second line.

\smallskip \noindent Suppose now that
$B\models\exists\bar y\,(H\bar y=\bar a+\bar s^T m)$
and that $\exists\bar y\,(H\bar y=K\bar x)\in\Sigma$. Proposition~\ref{preserve} gives
$D\models\exists\bar y\,(H\bar y=\bar a+\bar s^T n).$
Condition (2) of $\Sigma$ being pleasant and $C\leq_\Sigma D$, imply that the latter formula already holds in $C$. Thus every formula in $\operatorname{pp}_\Sigma(m/A;B)$ belongs to $\operatorname{pp}_\Sigma(n/A;C)$; the reverse inclusion is analogous.

\smallskip \noindent Conversely, suppose  that $\operatorname{pp}_{\Sigma}(m/A;B)=\operatorname{pp}_{\Sigma}(n/A;C)$. Let
$P=B\oplus_A C=(B\oplus C)/\widehat A$ with $\widehat A=\{(a,-a)\mid a\in A\}$. Let $D\leq P$ be the cyclic submodule generated by $[(m,-n)]$ in $P$. Let $Q=P/D$ and define
$f(b)=[(b,0)]+D$ and $g(c)=[(0,c)]+D.$
Then $f\restriction A=g\restriction A$ and $f(m)=g(n)$. We prove that $f$ is a $\preccurlyeq$-embedding; the proof for $g$ is analogous.

\smallskip \noindent First, we show that $f$ is a monormophism. If $f(b)=0$, then there are $s\in R$ and $a\in A$ such that
\[ b-sm=a \; \text{ and } \; sn=-a.\]
Thus $sz+a=0 \in \operatorname{pp}_\Sigma(n/A;C)$ and hence $ sz+a=0 \in  \operatorname{pp}_\Sigma(m/A;B)$. Therefore $sm+a=0$, and the first displayed equality gives $b=0$.

\smallskip \noindent It remains to show that $f[B]\leq_\Sigma Q$. Suppose
$Q\models\exists\bar y\,(H\bar y=Kf(\bar\ell)),$
where $\exists\bar y\,(H\bar y=K\bar x)\in\Sigma$ and $\bar\ell\in B^s$. Choose
$([(b_1,c_1)]+D,\ldots,[(b_n,c_n)]+D)\in Q^n$
witnessing the formula. Then there are $\bar a\in A^m$ and $\bar s\in R^m$ such that
\[
B\models H\bar b-K\bar\ell-\bar s^T m=\bar a,
\qquad
C\models H\bar c+\bar s^T n=-\bar a.
\]
Hence $\exists\bar y\,(H\bar y=-\bar a-\bar s^T z) \in \operatorname{pp}_\Sigma(n/A;C) = \operatorname{pp}_\Sigma(m/A;B)$. Choose $\bar e\in B^n$ such that
$B\models H\bar e=-\bar a-\bar s^T m.$
Adding the last two equations in $B$ we have that
$B\models H(\bar b+\bar e)=K\bar\ell.$
Therefore, $Q\models Hf(\bar b+\bar e)=Kf(\bar\ell)$ by Proposition \ref{preserve}. Hence $f[B]\leq_\Sigma Q$.\end{proof}

\begin{remark}
When $\preccurlyeq=\pleq$ and $\Sigma$ is the set from Example~\ref{points-(I)}(3), the characterization of Theorem \ref{char-gtps} already appears in \cite[Lemma~3.14]{kuma}. Nevertheless, the proof above is more direct as it avoids $\mrm{pp}$-quantifier elimination.
\end{remark}

\begin{remark}
The proof of Theorem \ref{char-gtps} extends the methods used in \cite[Theorem 2.16]{fgm1}.
\end{remark}

The preceding syntactic characterization of types yields tameness.
\begin{corollary}\label{tame-(I)}
	If $\preccurlyeq$ is pleasant, then $(\rmod, \preccurlyeq)$ is $(<\aleph_0)$-tame.
	\end{corollary}
	
	Furthermore, we use Theorem~\ref{char-gtps} to obtain a direct proof of stability.

\begin{proposition}\label{small-pp-types}
If $A$ is a left $R$-module, then
\[
|S_{\mathrm{pp}}(A)|\leq |A^{<\omega}|^{\operatorname{card}(R)+\aleph_0},
\]
where $S_{\mathrm{pp}}(A)=\{\operatorname{pp}(m/A;B)\mid A\pleq B,\ m\in B\}.$
\end{proposition}
\begin{proof}
Fix a well-ordering of $A^{<\omega}$, let $d\notin A^{<\omega}$, and let $\Gamma$ be the set of all $\mrm{pp}$-formulas of the form $\varphi(z,\bar x)$.

\smallskip \noindent Let
$\Psi:S_{\mathrm{pp}}(A)\longrightarrow(A^{<\omega}\cup\{d\})^\Gamma$ be given by $\Psi(\mrm{pp}(m/A, B)): \Gamma \to A^{<\omega} \cup \{d\}$ where for $\varphi(z,\bar x)\in\Gamma$, let $\Psi(\mrm{pp}(m/A, B))(\varphi)$ be the least tuple $\bar a\in A^{|\bar x|}$ such that $B\models\varphi(m,\bar a)$, if such a tuple exists, and let $\Psi(\mrm{pp}(m/A, B))(\varphi)=d$ otherwise.

\smallskip \noindent We show that $\Psi$ is injective. Suppose
$\Psi(\operatorname{pp}(m_1/A;B_1)) = \Psi(\operatorname{pp}(m_2/A;B_2)).$
It suffices to prove one inclusion. Let $\varphi(z,\bar a)\in\operatorname{pp}(m_1/A;B_1)$, and write $\bar a_\varphi=\Psi(\operatorname{pp}(m_1/A;B_1))(\varphi)$. Then
$B_1\models\exists z\bigl(\varphi(z,\bar a)\wedge\varphi(z,\bar a_\varphi)\bigr).$
Since $A\pleq B_1,B_2$ and the previous formula is a first-order $\mrm{pp}$-formula, $B_2 \models\exists z\bigl(\varphi(z,\bar a)\wedge\varphi(z,\bar a_\varphi)\bigr).$ As solution sets of $\mrm{pp}$-formulas are cosets \cite[Corollary~2.2]{prest_book1}, hence
$\varphi(B_2,\bar a)=\varphi(B_2,\bar a_\varphi)$. Since $\bar{a}_\varphi= \Psi(\mrm{pp}(m_1/A, B_1))(\varphi)=\Psi(\mrm{pp}(m_2/A, B_2))(\varphi)$, $B_2\models\varphi(m_2,\bar a_\varphi)$, and therefore $B_2\models\varphi(m_2,\bar a)$.

\smallskip \noindent Finally, $|\Gamma|\leq\operatorname{card}(R)+\aleph_0$ as $\Gamma$ is a set of first-order formulas, which gives the stated bound.
\end{proof}

	\begin{lemma}\label{stable-(I)}
Assume that $\preccurlyeq$ is pleasant. If $\lambda^{\operatorname{card}(R)+\aleph_0}=\lambda$, then $(\rmod,\preccurlyeq)$ is $\lambda$-stable.
\end{lemma}
\begin{proof}
Let $A\in (\rmod)_\lambda$ and $\preccurlyeq$ be pleasant with respect to $\Sigma$. Suppose, toward a contradiction, that
$\{p_i=\gtp_{(\rmod,\preccurlyeq)}(m_i/A;B)\mid i<\lambda^+\}$
is a family of distinct types in $\gS_{(\rmod,\preccurlyeq)}(A)$. By amalgamation, we may assume that all the types are realized in the same module $B$.

\smallskip \noindent Choose $\widehat A\pleq B$ such that $A\subseteq\widehat A$ and $|\widehat A|\leq\lambda$. Proposition~\ref{small-pp-types} and the hypothesis on $\lambda$ give that
\[
|S_{\mathrm{pp}}(\widehat A)|
\leq |\widehat A^{<\omega}|^{\operatorname{card}(R)+\aleph_0}
\leq\lambda.
\]
Hence there are distinct $i,j<\lambda^+$ with
$\operatorname{pp}(m_i/\widehat A;B)=\operatorname{pp}(m_j/\widehat A;B).$  Hence $\mrm{pp}_\Sigma(m_i/A, B) = \mrm{pp}_\Sigma(m_j/A, B)$ as the formulas that appear in these types are pp-formulas. Therefore, we have that
$$\gtp_{(\rmod, \preccurlyeq)}(m_i/A, B) = \gtp_{(\rmod,\preccurlyeq)}(m_j/A, B)$$ by Theorem \ref{char-gtps}. This is a contradiction to $p_i \neq p_j$.
\end{proof}
	
	\begin{remark}\label{gen-sta-1}
It follows from Lemma~\ref{stable-(I)} and Example~\ref{points-(I)} that the AEC $(\rmod,\preccurlyeq)$ is $\lambda$-stable whenever $\lambda^{\operatorname{card}(R)+\aleph_0}=\lambda$ and $\preccurlyeq\in\{\leq,\leq_{\mathrm{RD}},\pleq\}$. These three results were first obtained in \cite[Lemma~3.6]{maz1}, \cite[Theorem~3.11]{mj}, and \cite[Theorem~3.16]{kuma}, respectively.
	\end{remark}
	
	\begin{question}\label{q-first}
Let $\Sigma$ be a set of first-order $\mrm{pp}$-formulas. Do Lemma~\ref{AP-(I)}, Corollary~\ref{tame-(I)}, and Lemma~\ref{stable-(I)} hold for $(\rmod,\leq_\Sigma)$?
	\end{question}

\subsection{The sublattice $\mathscr{L}^{2}_{R}$}

We now study the sublattice given by the strong submodel relations on $\rmod$ which are stronger than the pure submodule relation.

\begin{definition}
Given a ring $R$, let $\mathscr{L}^{2}_{R}$ be the sublattice of $\mathscr{L}_{R}$ above the pure submodule relation.
\end{definition}

The structure of $\mathscr{L}^{2}_{R}$ again depends strongly on the ring $R$. For instance, if $R$ is left pure semisimple, then  purity coincides with the direct summand relation \cite[Theorem~4.5.7]{prest_book2}, so $\mathscr{L}^{2}_{R}$ has a single  element. In this section, we exhibit points of $\mathscr{L}^{2}_{R}$ and show that many of them have good stability-theoretic behavior. In Sections~\ref{5.2} and~\ref{5.3}, we further investigate $\mathscr{L}_{\mathbb{Z}}$.

\begin{proposition}
$\mathscr{L}^{2}_{R}$ is a lattice with bottom element given by $\pleq$ 
\end{proposition}

Our main tool for finding new points in $\mathscr{L}^{2}_{R}$ is the use of infinitary $\mrm{pp}$-formulas.

\begin{definition}\label{or_infty}
Let $\lambda\leq\kappa$ be infinite cardinals, let $\alpha$ be an ordinal, and let $A,B\in\rmod$. We say that $A$ is $(\kappa,\lambda,\alpha)$-pure in $B$, denoted by $A\leq^{(\kappa,\lambda,\alpha)}_{\mathrm{pp}}B$,
if $A\leq_{\Lambda^{\kappa,\lambda}_\alpha}B$. Equivalently, $A\leq B$ and, for every $\varphi(\bar x)\in\Lambda^{\kappa,\lambda}_\alpha$ and every $\bar a\in A^{|\bar x|}$,
\[
A\models\varphi(\bar a)\quad\Longleftrightarrow\quad B\models\varphi(\bar a).
\]
When $\lambda=\aleph_0$, we write $\leq^{(\kappa,\alpha)}_{\mathrm{pp}}$ instead of $\leq^{(\kappa,\aleph_0,\alpha)}_{\mathrm{pp}}$.

Similarly, we write $A\leq^{(\kappa,\lambda,\infty)}_{\mathrm{pp}}B$ if $A\leq_{\Lambda^{\kappa,\lambda}_\infty}B$. When $\lambda=\aleph_0$, we write $\leq^{(\kappa,\infty)}_{\mathrm{pp}}$ instead of $\leq^{(\kappa,\aleph_0,\infty)}_{\mathrm{pp}}$.
\end{definition}

\begin{proposition}\label{equiv-order}\
\begin{enumerate}[(1)]
\item $A\pleq B$ if and only if $A\leq^{(\aleph_0,\aleph_0,\omega)}_{\mathrm{pp}}B$ if and only if $A\leq^{(\aleph_0,\aleph_0,1)}_{\mathrm{pp}}B$.
\item For every regular cardinal $\kappa$, $A\leq^{(\kappa,\infty)}_{\mathrm{pp}}B$ if and only if $A\leq^{(\kappa,\kappa)}_{\mathrm{pp}}B$.
\item For every singular cardinal $\kappa$, $A\leq^{(\kappa,\infty)}_{\mathrm{pp}}B$ if and only if $A\leq^{(\kappa,\kappa^+)}_{\mathrm{pp}}B$.
\end{enumerate}
\end{proposition}
\begin{proof}
Part~(1) follows from Remark~\ref{equiv-remark}, while parts~(2) and~(3) follow from Proposition~\ref{equiv-kappa}.
\end{proof}

\begin{remark}
By Corollary~\ref{ref_sums}, every relation introduced in Definition~\ref{or_infty} refines the direct summand relation.
\end{remark}

For a cardinal $\theta$, let $\theta^-:=\theta$ if $\theta$ is a limit cardinal, and let $\theta^-:=\mu$ if $\theta=\mu^+$.

\begin{lemma}\label{prop_AEC_general_infinitary}
For every cardinal $\kappa$ and every ordinal $\alpha$, $(\rmod,\leq^{(\kappa,\alpha)}_{\mathrm{pp}})$ is an AEC and
\[
\mathrm{LS}(\rmod,\leq^{(\kappa,\alpha)}_{\mathrm{pp}})
\leq 2^{(\kappa +\operatorname{card}(R)+\aleph_0)^-}.
\]
Moreover, $(\rmod,\leq^{(\kappa,\infty)}_{\mathrm{pp}})$ is an AEC and
$\mathrm{LS}(\rmod,\leq^{(\kappa,\infty)}_{\mathrm{pp}})
\leq 2^{(\kappa  +\operatorname{card}(R)+\aleph_0)^-}.$
\end{lemma}
\begin{proof}
The \emph{moreover part} follows from Proposition~\ref{equiv-order}(2) and~(3), so it is enough to prove the main assertion. Fix a cardinal $\kappa$ and an ordinal $\alpha$. It is clear that $(\rmod,\leq^{(\kappa,\alpha)}_{\mathrm{pp}})$ is an abstract class. We verify Conditions~(4)--(6) of Definition~\ref{def_AEC}.

\begin{enumerate}[(4)]
\item[(4)] Let $(A_i)_{i<\delta}$ be an increasing continuous $\leq^{(\kappa,\alpha)}_{\mathrm{pp}}$-chain, and let $A:=\bigcup_{i<\delta}A_i$. Condition~(4.1) is clear. We prove Condition~(4.2); the proof of Condition~(4.3) is similar.

\smallskip \noindent We show by induction on $\beta\leq\alpha$ that, for every $i<\delta$, every $\varphi(\bar x)\in\Lambda^{\kappa,\aleph_0}_\beta$, and every $\bar a\in A_i^{|\bar x|}$,
$A\models\varphi(\bar a)$ if and only if $A_i\models\varphi(\bar a).$

\smallskip \noindent If $\beta=0$, the assertion follows because every formula in $\Lambda^{\kappa,\aleph_0}_0$ is a linear equation. If $\beta$ is a limit ordinal, the assertion follows directly from the induction hypothesis. Suppose that $\beta=\gamma+1$, fix $i<\delta$, let $\varphi(\bar x)\in\Lambda^{\kappa,\aleph_0}_{\gamma+1}$, and let $\bar a\in A_i^{|\bar x|}$.

\smallskip \noindent By Proposition~\ref{preserve}, $A_i\models\varphi(\bar a)$ implies $A\models\varphi(\bar a)$. Conversely, suppose that $A\models\varphi(\bar a)$. Then
\[
\varphi(\bar x)=\exists\bar y\bigwedge_{\psi\in\Psi}\psi(\bar x,\bar y),
\]
for some $|\bar y|<\aleph_0$ and some $\Psi\subseteq\Lambda^{\kappa,\aleph_0}_\gamma$ with $|\Psi|<\kappa$. Choose a finite tuple $\bar b\in A$ such that
$A\models\bigwedge_{\psi\in\Psi}\psi(\bar a,\bar b).$
There is $j\geq i$ such that $\bar a,\bar b\in A_j$. By the induction hypothesis,
$A_j\models\bigwedge_{\psi\in\Psi}\psi(\bar a,\bar b),$
so $A_j\models\varphi(\bar a)$. Since $A_i\leq^{(\kappa,\alpha)}_{\mathrm{pp}}A_j$, we conclude that $A_i\models\varphi(\bar a)$.

\item[(5)] Coherence is immediate.

\item[(6)] Since $\Lambda^{\kappa,\aleph_0}_\alpha\subseteq\mathfrak{L}_{\kappa,\omega}(\tau_R)$, the downward L\"owenheim--Skolem theorem for $\mathfrak{L}_{\kappa,\omega}(\tau_R)$ yields
the desired inequality (se for example \cite[Theorem 1.23]{mar}).
\end{enumerate}
\end{proof}

We next introduce a point of $\mathscr{L}^{2}_{R}$ that admits both an infinitary $\mrm{pp}$-characterization and an algebraic characterization. In Section~\ref{5.3}, we show that, over $\mathbb Z$, it acts as a \emph{dividing point} for the amalgamation property.

\begin{definition}\label{def_weak_pp}
A \emph{weak infinitary $\mrm{pp}$-formula} is a formula of the form
\[
\exists\bar y\bigwedge_{i<\kappa}\psi_i(\bar x,\bar y),
\]
with $|\bar x|,|\bar y|<\aleph_0$, $\kappa$ is a cardinal, and each $\psi_i(\bar x,\bar y)$ is a first-order $\mrm{pp}$-formula. We denote the set of all weak $\mrm{pp}$-formulas by $\Lambda^{\mathrm{wpp}}_\infty$.

\smallskip \noindent For $A,B\in\rmod$, we define $A\pleqq B$ if and only if $A\leq_{\Lambda^{\mathrm{wpp}}_\infty}B$. Equivalently, $A\leq B$ and, for every $\varphi(\bar x)\in\Lambda^{\mathrm{wpp}}_\infty$ and every $\bar a\in A^{|\bar x|}$,
\[
A\models\varphi(\bar a)\quad\Longleftrightarrow\quad B\models\varphi(\bar a).
\]
\end{definition}

\begin{remark}
Suppose that there is a finitely generated $R$-module $A$ which is not pure injective. Then $\pleqq$ is strictly stronger than $\pleq$. 

Indeed, let $B$ be the pure injective envelope of $A$ and fix a finite generating tuple $\bar a$ for $A$. We have that $A\pleq B$ and $A\lneq B$. Since $A$ is not pure-injective, there is a $\mrm{pp}$-type over $A$ which is finitely satisfiable in $A$ but not realized in $A$ (see for example \cite[Theorem 2.8]{prest_book1}). As every parameter from $A$ is an $R$-linear combination of $\bar a$, we may write this type as $p(y,\bar a)$. Since $B$ is pure-injective, $p(y,\bar a)$ is realized in $B$. Hence
\[
B\models\exists y\bigwedge_{\varphi \in p} \varphi( y,\bar a)
\quad\text{but}\quad
A\not\models\exists y\bigwedge_{\varphi \in p} \varphi( y,\bar a).
\]
The displayed formula is a weak infinitary $\mrm{pp}$-formula, so $A\not\pleqq B$. 

In particular, the hypothesis holds for $R=\mathbb Z$, taking $A=\mathbb Z$.
\end{remark}

\begin{remark}\label{rmk-compare}
There are at most $\operatorname{card}(R)+\aleph_0$ first-order $\mrm{pp}$-formulas in a finite number of variables $\bar x,\bar y$. Hence every $\varphi(\bar x)\in\Lambda^{\mathrm{wpp}}_\infty$ is equivalent, modulo the theory of $R$-modules, to a formula in $\Lambda^{(\operatorname{card}(R)+\aleph_0)^+,\aleph_0}_2$. In particular,
$\leq^{((\operatorname{card}(R)+\aleph_0)^+,\infty)}_{\mathrm{pp}} \subseteq\pleqq,$
and $\pleqq$ refines the direct summand relation.
\end{remark}

\begin{lemma}
The class $(\rmod,\pleqq)$ is an AEC with
\[
\mathrm{LS}(\rmod,\pleqq)\leq 2^{\operatorname{card}(R)+\aleph_0}.
\]
\end{lemma}
\begin{proof}
The proof is similar to  Lemma \ref{prop_AEC_general_infinitary} where $\mathrm{LS}(\rmod, \pleqq)\leq 2^{\mrm{card}(R)+\aleph_0}$ follows from Remark  \ref{rmk-compare}.
\end{proof}

We use the following definition from \cite{ziegler}.

\begin{definition}
Let $A$ and $B$ be modules, and let $Z\subseteq A$. A function $f:Z\to B$ is a \emph{$\mrm{pp}$-$(A,B)$-homomorphism} if, for every first-order $\mrm{pp}$-formula $\varphi(\bar x)$ and every $\bar z\in Z^{|\bar x|}$,
\[
A\models\varphi(\bar z)\quad\Longrightarrow\quad B\models\varphi(f(\bar z)).
\]
\end{definition}

\begin{lemma}
Let $A,B\in\rmod$. The following are equivalent:
\begin{enumerate}[(1)]
\item $A\pleqq B$.
\item $A\leq B$ and, for every $\bar a\in A^{<\omega}$ and $\bar b\in B^{<\omega}$, there is a $\mrm{pp}$-$(B,A)$-homomorphism $f$ whose domain contains the entries of $\bar a\bar b$, such that $f(\bar a)=\bar a$ and $f(\bar b)\in A^{|\bar b|}$.
\end{enumerate}
\end{lemma}
\begin{proof}
Suppose first that $A\pleqq B$. Fix $\bar a\in A^{<\omega}$ and $\bar b\in B^{<\omega}$, and let $p=\operatorname{pp}(\bar b/\bar a;B)$. Then
$B\models\bigwedge_{\varphi\in p}\varphi(\bar b,\bar a),$
and hence
$B\models\exists\bar y \bigwedge_{\varphi\in p}\varphi(\bar y,\bar a).$
Since $A\pleqq B$, there is $\bar c\in A^{|\bar b|}$ such that
$A\models\bigwedge_{\varphi\in p}\varphi(\bar c,\bar a).$
The function on the entries of $\bar a\bar b$ that fixes $\bar a$ and sends $\bar b$ to $\bar c$ is therefore a $\mrm{pp}$-$(B,A)$-homomorphism.

\smallskip \noindent Conversely, let
\[
\varphi(\bar x)=\exists\bar y\bigwedge_{i<\kappa}\psi_i(\bar x,\bar y),
\]
let $\bar a\in A^{|\bar x|}$, and suppose that $B\models\varphi(\bar a)$. Choose $\bar b\in B^{|\bar y|}$ such that
$B\models\bigwedge_{i<\kappa}\psi_i(\bar a,\bar b).$
By assumption, there is a $\mrm{pp}$-$(B,A)$-homomorphism $f$ that fixes $\bar a$ and sends $\bar b$ into $A$. Thus
$A\models\bigwedge_{i<\kappa}\psi_i(\bar a,f(\bar b)),$
so $A\models\varphi(\bar a)$. The reverse implication follows from Proposition~\ref{preserve}, and hence $A\pleqq B$.
\end{proof}

We now show that stability is pervasive in $\mathscr{L}^{2}_{R}$. This contrasts with amalgamation, which, as we will show in Section~\ref{5.3}, often fails.

\begin{definition}\label{pre_def_nice}
A set $\Sigma\subseteq\Lambda^{\lambda^+,\lambda^+}_1$, where $\lambda$ is an infinite cardinal with $\operatorname{card}(R)\leq\lambda$, is \emph{E-nice} if every $\varphi(\bar x)\in\Sigma$ has the form
\[
\exists\bar v\bigwedge_{i<\lambda}\delta_i(\bar v,\bar x)=0,
\]
with $\bar x=(x_0,\ldots,x_{n-1})$, $\bar v=(v_\ell)_{\ell<\lambda}$ and
\[
\delta_i(\bar v,\bar x)
\doteq
\sum_{\ell<\lambda}r_{i,\ell}v_\ell+
\sum_{j<n}q_{i,j}x_j,
\]
with $r_{i,\ell},q_{i,j}\in R$ and for each $i<\lambda$, $r_{i,\ell} \neq 0$ for only finitely many $\ell < \lambda$.

\smallskip \noindent A binary relation $\preccurlyeq$ on $\rmod$ is \emph{E-nice with respect to $\Sigma$} if $\Sigma$ is E-nice and $\preccurlyeq=\leq_\Sigma$. We say  $\preccurlyeq$ is \emph{E-nice} if it is E-nice with respect to some such $\Sigma$.
\end{definition}

\begin{remark}\label{card-sigma}\
\begin{enumerate}[$\bullet$]
\item It is important to note that if $\preccurlyeq$ is \emph{E-nice}, we do not necessarily have that $\preccurlyeq \in \mathscr{L}_{R}$.
\item  If $\Sigma \subseteq \Lambda^{\lambda^+, \lambda^+}_1$ is E-nice, then $|\Sigma| \leq 2^\lambda$  by the syntactic structure of the formulas in $\Sigma$.
\end{enumerate}
\end{remark}

\begin{lemma}\label{first_nice_lemma}
If
$\, \preccurlyeq\in
\{\leq^{(\kappa,\alpha)}_{\mathrm{pp}}\mid \kappa\in\mathrm{Card},\ \alpha\in\mathrm{Ord}\}
\cup\{\leq^{(\kappa,\infty)}_{\mathrm{pp}}\mid \kappa\in\mathrm{Card}\}
\cup\{\pleqq\},$
then $\preccurlyeq$ is E-nice.
\end{lemma}
\begin{proof}
For $\prec \, \in\{\leq^{(\kappa, \alpha)}_{\mathrm{pp}}\, \mid \,  \kappa \in \mrm{Card} \, ,  \alpha \in \mrm{Ord}\}$ it follows from Proposition \ref{equiv-exist}. For $\prec \, \in\{\leq^{(\kappa, \infty)}_{\mathrm{pp}} \, \mid \, \kappa \in \mrm{Card} \}$ it follows from Propositions \ref{equiv-order} and \ref{equiv-exist}. For $\prec = \pleqq$ it follows from Remark \ref{rmk-compare} and Proposition \ref{equiv-exist}.
\end{proof}

\begin{definition}
Suppose that
\[
\varphi(\bar x):=\exists\bar v\bigwedge_{i<\lambda}\delta_i(\bar v,\bar x)=0,
\]
with $\bar x=(x_0,\ldots,x_{n-1})$, $\bar v=(v_\ell)_{\ell<\lambda}$ and
\[
\delta_i(\bar v,\bar x)
\doteq
\sum_{\ell<\lambda}r_{i,\ell}v_\ell+
\sum_{j<n}q_{i,j}x_j.
\]

with $r_{i,\ell},q_{i,j}\in R$ and for each $i<\lambda$, $r_{i,\ell} \neq 0$ for only finitely many $\ell < \lambda$. 

Let $\varphi^*(\bar x',\bar x'')$ be the formula
\[
\exists\bar v\bigwedge_{i<\lambda}\delta_i^*(\bar v,\bar x',\bar x'')=0,
\]
with $\bar x'=(x_0,\ldots,x_{n-1})$, $\bar x''=(x_n,\ldots,x_{2n-1})$, and
\[
\delta_i^*(\bar v,\bar x',\bar x'')
\doteq
\sum_{\ell<\lambda}r_{i,\ell}v_\ell+
\sum_{j<n}q_{i,j}x_j+
\sum_{j<n}q_{i,j}x_{n+j}.
\]
\end{definition}

\begin{proposition}\label{easy-nice}
Assume that $\preccurlyeq$ is E-nice with respect to $\Sigma$ and that $\mathcal K=(\mathbf K,\preccurlyeq)$ is an AEC with $\mathbf K \subseteq\rmod$. If $\varphi(\bar x)\in\Sigma$, $A\preccurlyeq B$ with $A,B\in\mathbf K$, and $\bar a,\bar a'\in A^{|\bar x|}$, then
$A\models\varphi^*(\bar a,\bar a')$ if and only if $B\models\varphi^*(\bar a,\bar a').$
\end{proposition}
\begin{proof}
The forward implication follows from Proposition~\ref{preserve}. Conversely, suppose that $B\models\varphi^*(\bar a,\bar a')$. Then $B\models\varphi(\bar a+\bar a')$. Since $\bar a+\bar a'\in A^{|\bar x|}$, $\varphi(\bar x) \in\Sigma$, and $A\leq_\Sigma B$, we have that $A\models\varphi(\bar a+\bar a')$. Hence $A\models\varphi^*(\bar a,\bar a')$.
\end{proof}

If $A,C\leq B$, let $\langle AC\rangle_B$ denote the submodule of $B$ generated by $A\cup C$. Thus
$\langle AC\rangle_B=A+C=\{a+c\mid a\in A,\ c\in C\}.$

\begin{remark}\label{star-rmk}
Assume that $\preccurlyeq$ is E-nice with respect to $\Sigma$ and that $\mathcal K=(\mathbf K,\preccurlyeq)$ is an AEC with $\mathbf K \subseteq\rmod$. Suppose that $A,C\leq B$, that $\langle AC\rangle_B,B\in\mathbf K$, and that $\langle AC\rangle_B\not\preccurlyeq B$. Then there are $\varphi(\bar x)\in\Sigma$, with $|\bar x|=n<\omega$, and $\bar d\in\langle AC\rangle_B^n$ such that
\[
B\models\varphi(\bar d)
\quad\text{and}\quad
\langle AC\rangle_B\not\models\varphi(\bar d).
\]
Writing $\bar d=\bar a+\bar c$, with $\bar a\in A^n$ and $\bar c\in C^n$, we obtain
\[
B\models\varphi^*(\bar a,\bar c)
\quad\text{and}\quad
\langle AC\rangle_B\not\models\varphi^*(\bar a,\bar c).
\]
\end{remark}

\begin{theorem}\label{sta-nice}
Assume that $\preccurlyeq$ is E-nice. If $\mathbf K\subseteq\rmod$ satisfies the following conditions:
\begin{enumerate}[(1)]
\item $\mathcal K=(\mathbf K,\preccurlyeq)$ is an AEC;
\item $\mathbf K$ is closed under $R$-submodules, i.e.,  if $A\leq B$ and $B\in\mathbf K$, then $A\in\mathbf K$.
\end{enumerate}
Then $\mathcal K$ is stable.
\end{theorem}
\begin{proof}
Let $\Sigma\subseteq\Lambda^{\lambda^+,\lambda^+}_1$ E-nice for $\lambda$ a cardinal such that $\preccurlyeq=\leq_\Sigma$. A standard witness closing construction, using  that $|\Sigma|\leq 2^\lambda$ from Remark~\ref{card-sigma} and the closure of $\mathbf K$ under submodules, gives $\mathrm{LS}(\mathcal K)\leq 2^\lambda$.

\smallskip \noindent We prove the theorem through three claims.

\smallskip \noindent \underline{Claim 1.} If $A\preccurlyeq B$ and $b\in B$, where $A,B\in\mathbf K$, then there is $C\in\mathbf K$ such that
\begin{enumerate}[(a)]
\item $C\preccurlyeq B$ and $|C|\leq 2^\lambda$;
\item $b\in C$;
\item $\langle AC\rangle_B\preccurlyeq B$.
\end{enumerate}

\smallskip \noindent \emph{Proof of Claim 1.} Suppose, toward a contradiction, that no such $C$ exists. In particular, $|B|>2^\lambda$, since otherwise one could take $C=B$.

\smallskip \noindent We recursively construct an increasing continuous $\preccurlyeq$-chain $(C_i)_{i<(2^\lambda)^+}$ in $\mathbf{K}$, formulas $(\varphi_i(\bar x_i))_{i <(2^\lambda)^+}$ in $\Sigma$, and tuples $((\bar{a}_i, \bar{c}_i))_{i < (2^\lambda)^+}$, so that for every $i<(2^\lambda)^+$:
\begin{enumerate}[(1)]
\item $|C_i|=2^\lambda$;
\item $b\in C_0$ and  $C_i\preccurlyeq B$;
\item  $\bar a_i\in A^{|\bar x_i|}$ and $\bar c_i\in C_i^{|\bar x_i|}$;
\item $\langle AC_{i+1}\rangle_B\models\varphi_i^*(\bar a_i,\bar c_i)$ and $\langle AC_i\rangle_B\not\models\varphi_i^*(\bar a_i,\bar c_i)$.
\end{enumerate}

\smallskip \noindent The construction follows from the L\"owenheim--Skolem--Tarski axiom and Remark~\ref{star-rmk}. At a successor step, we make sure  $C_{i+1}$ has the at most $\lambda$ witnesses for the existential quantifiers in $\varphi_i^*(\bar x_i',\bar x_i'')$. It is in this construction where we use in a key way the assumption that $\mathbf{K}$ is closed under submodules.

\smallskip \noindent Let $X_1$ be the set of limit ordinals below $(2^\lambda)^+$. For each $i\in X_1$, let $\alpha_i<i$ be least such that $\bar c_i\in C_{\alpha_i}^{|\bar x_i|}$; such an $\alpha_i$ exists because $\bar c_i$ is finite and the chain is continuous. By Fodor's lemma, there are a stationary set $X_2\subseteq X_1$ and an ordinal $\alpha$ such that $\alpha_i=\alpha$ for every $i\in X_2$. 

\smallskip \noindent  Let $\Phi: X_2 \to \Sigma \times C_{\alpha}^{<\omega}$ given by $\Phi(i) = (\varphi_i(\bar{x}_i), \bar{c}_i)$. The map has range of cardinality at most $2^\lambda$ by Remark \ref{card-sigma}. Hence, it follows from the pigeonhole principle that there are a set $X\subseteq X_2$ of cardinality $(2^\lambda)^+$, a formula $\varphi(\bar x)\in\Sigma$, and a tuple $\bar c\in C_\alpha^{|\bar x|}$ such that
$\varphi_i(\bar x_i)=\varphi(\bar x)$ and $\bar c_i=\bar c$
for every $i\in X$.

\smallskip \noindent Choose $i<j$ both in $X$. By Condition~(4),
\[
\langle AC_{i+1}\rangle_B\models\varphi^*(\bar a_i,\bar c)
\quad\text{and}\quad
\langle AC_{j+1}\rangle_B\models\varphi^*(\bar a_j,\bar c).
\]
Since $i<j$, Proposition~\ref{preserve}(1) gives
$\langle AC_{j+1}\rangle_B\models\varphi^*(\bar a_i,\bar c).$
By Proposition~\ref{preserve}(2),
$\langle AC_{j+1}\rangle_B\models \varphi^*(\bar{a}_j - \bar{a}_i, \bar{c}-\bar{c})= \varphi^*(\bar{a}_j - \bar{a}_i, \bar{0}),$
and hence $B\models\varphi^*(\bar a_j-\bar a_i,\bar 0)$. Since $A\preccurlyeq B$, Proposition~\ref{easy-nice} yields
$A\models\varphi^*(\bar a_j-\bar a_i,\bar 0).$
Combining this with $\langle AC_{i+1}\rangle_B\models\varphi^*(\bar a_i,\bar c)$ and applying Proposition~\ref{preserve}(1) and~(2), we obtain
$\langle AC_{i+1}\rangle_B\models\varphi^*(\bar a_j,\bar c).$
Because $i+1<j$ as $j$ is a limit ordinal, Proposition~\ref{preserve}(1) implies
$\langle AC_j\rangle_B\models\varphi^*(\bar a_j,\bar c_j),$
contradicting Condition~(5). This proves Claim~1.

\smallskip \noindent \underline{Claim 2.} Suppose that $A\preccurlyeq B$ and $A\preccurlyeq B'$, with $b\in B$ and $b'\in B'$. Let $C\preccurlyeq B$ and $C'\preccurlyeq B'$ satisfy Conditions~(a)--(c) of Claim~1 for $b \in B$ and $b' \in B'$, respectively. If $A\cap C=A\cap C'$ and there is an isomorphism $f:C\to C'$ such that
\[
f\restriction(A\cap C)=\operatorname{id}_{A\cap C}
\quad\text{and}\quad
f(b)=b',
\]
then
$\gtp_{\mathcal K}(b/A;B)=\gtp_{\mathcal K}(b'/A;B').$

\smallskip \noindent \emph{Proof of Claim 2.} Let $P_C:= A \oplus_{A \cap C} C$ be the pusout and $t_C: P_C \to \langle AC \rangle_B$ be the map given by the universal mapping property, i.e., $t_C[(a,c)]_{A \cap C} = a + c$.

The map $t_C$ is an epimorphism, we show it is a monomorphism. If $t_C([(a,c)]_{A\cap C})=0$, then $a=-c\in A\cap C$, and therefore $[(a,c)]_{A\cap C}=0$. Thus $t_C$ is an isomorphism. Its inverse is $g_C:\langle AC\rangle_B\longrightarrow P_C$ given by $g_C(a+c)=[(a,c)]_{A\cap C}.$

Similarly, the map $t_{C'}:A\oplus_{A\cap C'}C'\longrightarrow\langle AC'\rangle_{B'}$ given by $t_{C'}([(a,c')]_{A\cap C'})=a+c'$,
is an isomorphism.

\smallskip \noindent Let $h:A\oplus_{A\cap C}C\longrightarrow A\oplus_{A\cap C'}C'$ given by
$h([(a,c)]_{A\cap C})=[(a,f(c))]_{A\cap C'}.$ This is well defined because $A\cap C=A\cap C'$ and $f$ is the identity on this common submodule. Its inverse is induced by $f^{-1}$, so $h$ is an isomorphism. Consequently,
\[
s:=t_{C'}\circ h\circ g_C:
\langle AC\rangle_B\longrightarrow\langle AC'\rangle_{B'}
\]
is an isomorphism satisfying $s\restriction A=\operatorname{id}_A$ and $s(b)=b'$.

\smallskip \noindent Since $\mathbf K$ is closed under submodules, both $\langle AC\rangle_B, \langle AC'\rangle_{B'} \in \mathbf K$. So $(b, A, \langle AC \rangle_B) E^{\mrm{at}}_\mathcal{K} (b', A, \langle AC' \rangle_{B'})$.

 Moreover, Condition~(c) of Claim~1 gives
$\langle AC\rangle_B\preccurlyeq B$ and $\langle AC'\rangle_{B'}\preccurlyeq B'.$
Thus
\[
\begin{aligned}
(b,A,B)
&\ E^{\mathrm{at}}_{\mathcal K}\ (b,A,\langle AC\rangle_B)\\
&\ E^{\mathrm{at}}_{\mathcal K}\ (b',A,\langle AC'\rangle_{B'})\\
&\ E^{\mathrm{at}}_{\mathcal K}\ (b',A,B'),
\end{aligned}
\]
which proves Claim~2.

\smallskip \noindent \underline{Claim 3.} If $\theta^{2^\lambda}=\theta$, then $\mathcal K$ is $\theta$-stable.

\smallskip \noindent \emph{Proof of Claim 3.} Let $A\in\mathbf K_\theta$. Suppose, toward a contradiction, that
$\{p_i=\gtp_{\mathcal K}(b_i/A;B_i)\mid i<\theta^+\}$
is a family of distinct types in $\gS_{\mathcal{K}}(A)$. For each $i<\theta^+$, choose $C_i\preccurlyeq B_i$ satisfying Conditions~(a)--(c) of Claim~1 for $b_i\in B_i$.

\smallskip \noindent Since there are at most $\theta^{2^\lambda}=\theta$ subsets of $A$ of cardinality at most $2^\lambda$, by the pigeonhole principle there are a set $S\subseteq\theta^+$ of cardinality $\theta^+$ and a set $D\subseteq A$, with $|D|\leq2^\lambda$, such that
$A\cap C_i=D$
for every $i\in S$. Fix an enumeration $D=\{d_\alpha\mid\alpha<\mu\}$, where $\mu\leq2^\lambda$.

\smallskip \noindent Expand the language of $R$-modules by constants $k$ and $(e_\alpha)_{\alpha<\mu}$. Let $\Delta$ be the set of isomorphism types of structures
$(N,k^N,(e_\alpha^N)_{\alpha<\mu}),$
where $N\in\mathbf K$, $|N|\leq2^\lambda$, $k^N\in N$, and $e_\alpha^N\in N$ for every $\alpha<\mu$. Then
$|\Delta|\leq2^{2^\lambda}\leq\theta^{2^\lambda}=\theta.$
For $i\in S$, assign to $i$ the isomorphism type of
$(C_i,b_i,(d_\alpha)_{\alpha<\mu}).$
By the pigeonhole principle, there are distinct $i,j\in S$ and an isomorphism
\[
f:(C_i,b_i,(d_\alpha)_{\alpha<\mu})
\cong
(C_j,b_j,(d_\alpha)_{\alpha<\mu}).
\]
Thus $f: C_i \cong C_j$, $f(b_i)=b_j$ and $f$ fixes $D=A\cap C_i=A\cap C_j$ pointwise. Claim~2 gives $p_i=p_j$, a contradiction. This proves Claim~3.

\smallskip \noindent For every cardinal $\mu\geq2^\lambda$, the cardinal $\theta=2^\mu$ satisfies $\theta^{2^\lambda}=\theta$. Hence Claim~3 yields stability in unboundedly many cardinals, and $\mathcal K$ is stable.
\end{proof}

The following consequence is immediate from Lemma~\ref{first_nice_lemma} and Theorem~\ref{sta-nice}.

\begin{corollary}\label{many-sta}

If
$\, \preccurlyeq\in
\{\leq^{(\kappa,\alpha)}_{\mathrm{pp}}\mid \kappa\in\mathrm{Card},\ \alpha\in\mathrm{Ord}\}
\cup\{\leq^{(\kappa,\infty)}_{\mathrm{pp}}\mid \kappa\in\mathrm{Card}\}
\cup\{\pleqq\},$
then $(\rmod,\preccurlyeq)$ is stable.
\end{corollary}

We can strengthen Corollary~\ref{many-sta} as follows.

\begin{corollary}\label{more-sta}
Let $\Sigma\subseteq\Lambda^{\infty,\aleph_0}_\infty$ be a set. If $\leq_\Sigma\in\mathscr{L}_{R}$, then $(\rmod,\leq_\Sigma)$ is stable.
\end{corollary}
\begin{proof}
By Proposition~\ref{equiv-exist}, there is an E-nice set $\Sigma'$ such that $\leq_\Sigma=\leq_{\Sigma'}$. The conclusion now follows from Theorem~\ref{sta-nice}.
\end{proof}

\begin{corollary}\label{sta-first}
If $\Sigma$ is a set of first-order $\mrm{pp}$-formulas, then $(\rmod,\leq_\Sigma)$ is a stable AEC.
\end{corollary}
\begin{proof}
This follows from Proposition~\ref{AEC-plea} and Corollary~\ref{more-sta}.
\end{proof}

Thus Corollary~\ref{sta-first} partially answers Question~\ref{q-first} by extending Lemma~\ref{stable-(I)}.

\begin{remark}\label{gen-sta-2}
Theorem~\ref{sta-nice} also yields stability for other interesting AECs of modules. For example, it applies to the class of torsion abelian groups with pure embeddings and to the class of torsion-free abelian groups with pure embeddings, because both underlying classes are closed under submodules. These results were first obtained in \cite[Lemma~3.5]{maztor} and \cite[Theorem~0.3]{bet} respectively.
\end{remark}

\section{The lattice over $\mathbb{Z}$}

In this section, we study the lattice $\mathscr{L}_{\mathbb{Z}}$. The first subsection concerns $\mathscr{L}^{1}_{\mathbb{Z}}$; the second establishes that $\mathscr{L}_{\mathbb{Z}}$ is a proper class; and the third investigates the failure of amalgamation.

\subsection{\texorpdfstring{The sublattice $\mathscr{L}^{1}_{\mathbb{Z}}$}{The sublattice L1 over Z}}\label{5.1}

We begin with two elementary order-theoretic observations.

\begin{lemma}\label{antichain}
There is an uncountable antichain in $\mathscr{L}^{1}_{\mathbb{Z}}$.
\end{lemma}
\begin{proof}
Let $\mathcal A$ be an almost disjoint family of infinite subsets of the set of prime numbers such that $|\mathcal A|=2^{\aleph_0}$. Thus every $U\in\mathcal A$ is countably infinite, while $U\cap V$ is finite whenever $U,V\in\mathcal A$ are distinct. Such a family exists by for example \cite[Lemma~9.21]{jech}.

\smallskip \noindent 
For  $U \in \mathcal{A}$, let $A \leq_U B$ if and only if $A \leq B$ and for every $p \in U$, $pB \cap A = pA$. For every $U \in \mathcal{A}$, we have that $\leq_U$ is pleasant  with respect to $\Sigma_U =\{ \exists y (py = x) \, \mid \,  p \in U \}$  . Hence  $\leq_U \in  \mathscr{L}^{1}_{\mathbb{Z}}$ and $(\zmod,\leq_U)$ is $(<\aleph_0)$-tame and stable.

\smallskip \noindent We show that $\{\leq_U\mid U\in\mathcal A\}$ is an antichain. Let $U,V\in\mathcal A$ be distinct. Since $U$ and $V$ are infinite and $U\cap V$ is finite, neither is contained in the other. We prove that $\leq_V\not\subseteq\leq_U$; the reverse non-inclusion is analogous.

\smallskip \noindent Choose $p\in U\setminus V$, and let $G=\mathbb Z$ and $H= \langle \frac{1}{p} \rangle$. Clearly $G\not\leq_U H$. To see that $G\leq_V H$, fix $q\in V$ and $n\in qH\cap G$. Then there is $a\in\mathbb Z$  such that $
n=\frac{qa}{p}$,
so $p n=qa$. Since $p\neq q$, it follows that $q\mid n$ in $\mathbb Z$. Thus $n\in q\mathbb Z=qG$, and therefore $G\leq_V H$.
\end{proof}

The following natural question remains open.

\begin{question}
Is there an antichain of cardinality $2^{2^{\aleph_0}}$ in $\mathscr{L}^{1}_{\mathbb{Z}}$? This is the largest possible cardinality, since $
|\mathscr{L}^{1}_{\mathbb{Z}}|\leq 2^{2^{\aleph_0}}$
by Proposition~\ref{basic_L1}.
\end{question}

\begin{proposition}
There is a countable strictly increasing chain in $\mathscr{L}^{1}_{\mathbb{Z}}$.
\end{proposition}
\begin{proof}
Fix a prime $p$. For $n<\omega$, define $A\leq_{p,n}B$ if and only if $A\leq B$ and $
p^kB\cap A=p^kA$
for every $k\in\{0,\ldots,n\}$. The relation $\leq_{p,n}$ is pleasant with respect to $\Sigma_{p,n}=\{\exists y\,(p^ky=x)\mid 0\leq k\leq n\}$
and hence $\leq_{p,n}\in\mathscr{L}^{1}_{\mathbb Z}$. The relations $(\leq_{p,n})_{n<\omega}$ form a strictly increasing chain in $\mathscr{L}^{1}_{\mathbb Z}$.
\end{proof}

It is likewise open whether $\mathscr{L}^{1}_{\mathbb Z}$ contains an uncountable strictly increasing chain.

We next show that not all  the strong submodel relations in $\mathscr{L}_{\mathbb Z}$ are of the syntactic form so far considered in this paper.

\begin{definition}\label{def_syntactic}
A binary relation $\preccurlyeq$ on $\rmod$ is \emph{positive syntactic with respect to $\Sigma$} if $\Sigma\subseteq\Lambda^{\infty,\infty}_\infty$ and $\preccurlyeq=\leq_\Sigma$. We say $\preccurlyeq$  is \emph{positive syntactic} if it is positive syntactic with respect to some such $\Sigma$.
\end{definition}

Pleasant and E-nice relations are positive syntactic. Thus every relation introduced so far is positive syntactic. The next theorem shows that this need not hold for an arbitrary element of $\mathscr{L}_{\mathbb Z}$.

\begin{theorem}\label{n-syntac}
There is a relation $\preccurlyeq\in\mathscr{L}^{1}_{\mathbb Z}$ that is not positive syntactic.
\end{theorem}
\begin{proof}
For $A,B\in\zmod$, define $A\preccurlyeq B$ if and only if $A\leq B$ and, for every element $a\in A$ of infinite order and every integer $0<n<\omega$,
$A\models\exists y\,(ny=a)$ if and only if  $B\models\exists y\,(ny=a).$
It is straightforward to verify that $(\zmod,\preccurlyeq)$ is an AEC and that $\preccurlyeq$ lies below $\pleq$ in $\mathscr{L}_{\mathbb Z}$.

\smallskip \noindent Suppose, toward a contradiction, that $\preccurlyeq=\leq_\Sigma$ for some $\Sigma\subseteq\Lambda^{\infty,\infty}_\infty$. Fix a prime $p$, let $B=\mathbb Z/p^2\mathbb Z$, and let $A=pB$. Since $A$ has no elements of infinite order, $A\preccurlyeq B$, and therefore $A\leq_\Sigma B$. By applying Proposition~\ref{preserve} to the canonical inclusions and projections and adding, we obtain
$A\oplus\mathbb Z\leq_\Sigma B\oplus\mathbb Z.$
Hence $A\oplus\mathbb Z\preccurlyeq B\oplus\mathbb Z$.

\smallskip \noindent The element
$(p+p^2\mathbb Z,p)\in A\oplus\mathbb Z$
has infinite order. It is divisible by $p$ in $B\oplus\mathbb Z$, but it is not divisible by $p$ in $A\oplus\mathbb Z$. This contradicts that $A\oplus\mathbb Z\leq_\Sigma B\oplus\mathbb Z.$
\end{proof}

\begin{remark}
The relation from Theorem~\ref{n-syntac} does not have the amalgamation property. The span
\[
p(\mathbb Z/p^2\mathbb Z)\preccurlyeq
\mathbb Z/p^2\mathbb Z,
\quad
p(\mathbb Z/p^2\mathbb Z)\preccurlyeq
p(\mathbb Z/p^2\mathbb Z)\oplus\mathbb Z
\]
cannot be completed to a commutative square of strong embeddings. We do not know whether $(\zmod,\preccurlyeq)$ is stable.
\end{remark}

\subsection{\texorpdfstring{$\mathscr{L}_{\mathbb{Z}}$ is a proper class}{$L_{\mathbb Z}$ is a proper class}}\label{5.2}

We recall several classical notions from abelian group theory.

\begin{definition}[{\cite[p.~299]{fuchs}}]
Let $A\in\zmod$ and let $p$ be a prime. For every ordinal $\alpha$, define $p^\alpha A$ by induction as follows:
\begin{enumerate}[(1)]
\item $p^0A=A$;
\item $p^{\beta+1}A=p(p^\beta A)$;
\item if $\alpha$ is a limit ordinal, then $p^\alpha A=\bigcap_{\beta<\alpha}p^\beta A$.
\end{enumerate}
\end{definition}

\begin{definition}[{\cite[p.~386]{fuchs}}]
Let $\alpha$ be an ordinal and let $p$ be a prime. For $A,B\in\zmod$, we say that $A$ is \emph{$p^\alpha$-isotype in $B$}, denoted by $A\leq^{p^\alpha}_{\mathrm{iso}}B$, if $A\leq B$ and
$p^\beta B\cap A=p^\beta A$
for every $\beta\leq\alpha$.
\end{definition}

\begin{proposition}\label{iso_nice}
Let $\alpha$ be an ordinal and let $p$ be a prime. The relation $\leq^{p^\alpha}_{\mathrm{iso}}$ is E-nice with respect to a set
\[
\{\psi_\beta(x_0)\mid\beta\leq\alpha\}
\subseteq
\Lambda^{|\alpha|^++\aleph_0,\,|\alpha|^++\aleph_0}_1.
\]
Moreover, $\leq^{(|\alpha|^++\aleph_0,\infty)}_{\mathrm{pp}}
\subseteq
\leq^{p^\alpha}_{\mathrm{iso}}$, 
and $\leq^{p^\alpha}_{\mathrm{iso}}$ refines the direct summand relation.
\end{proposition}
\begin{proof}
We recursively define formulas $\varphi_\beta(x_0)$ for $\beta\leq\alpha$.

\smallskip \noindent Let $\varphi_0(x_0):=x_0=x_0$.

\smallskip \noindent If $\beta=\gamma+1$, let
\[
\varphi_{\gamma+1}(x_0):=\exists v\,(pv=x_0\wedge\varphi_\gamma(v)).
\]

\smallskip \noindent If $\beta$ is a limit ordinal, let
\[
\varphi_\beta(x_0):=\bigwedge_{\gamma<\beta}\varphi_\gamma(x_0).
\]

\smallskip \noindent Put $\Delta=\{\varphi_\beta(x_0)\mid\beta\leq\alpha\}$. Then $A\leq^{p^\alpha}_{\mathrm{iso}}B$ if and only if  $A\leq_\Delta B.$

\smallskip \noindent For every $\beta\leq\alpha$, $\varphi_\beta(x_0)\in \Lambda^{|\beta|^++\aleph_0,\aleph_0}_{\beta}$. Proposition~\ref{equiv-exist} therefore gives an E-nice formula
$\psi_\beta(x_0)\in \Lambda^{|\beta|^++\aleph_0,\,|\beta|^++\aleph_0}_1$
equivalent to $\varphi_\beta(x_0)$ for every $\beta \leq \alpha$. Hence $\leq^{p^\alpha}_{\mathrm{iso}}
=
\leq_{\{\psi_\beta(x_0)\mid\beta\leq\alpha\}}$, 
so $\leq^{p^\alpha}_{\mathrm{iso}}$ is E-nice.

\smallskip \noindent The \emph{moreover part} follows because each $\varphi_\beta(x_0) \in \Lambda^{|\alpha|^++\aleph_0,\aleph_0}_\infty$. The final assertion follows from Corollary~\ref{ref_sums}.
\end{proof}

\begin{lemma}\label{iso_AEC}
Let $\alpha$ be an ordinal and let $p$ be a prime. Then $(\zmod,\leq^{p^\alpha}_{\mathrm{iso}})$ is an AEC with
\[
\mathrm{LS}(\zmod,\leq^{p^\alpha}_{\mathrm{iso}})
\leq |\alpha|+\aleph_0.
\]
\end{lemma}
\begin{proof}
It is clear that $(\zmod,\leq^{p^\alpha}_{\mathrm{iso}})$ is an abstract class. Transitivity follows from Proposition~\ref{iso_nice}, or directly from \cite[p.~365, Condition~(b)]{fuchs}. We verify Conditions~(4)--(6) of Definition~\ref{def_AEC}.

\begin{enumerate}[(4)]
\item[(4)]  Let $(A_i)_{i<\delta}$ be an increasing continuous $\leq^{p^\alpha}_{\mathrm{iso}}$-chain, and let $A:=\bigcup_{i<\delta}A_i$. Condition~(4.1) is clear, and Condition~(4.2) is \cite[p.~365, Condition~(c)]{fuchs}. To verify Condition~(4.3), suppose that $A_i\leq^{p^\alpha}_{\mathrm{iso}}B$ for every $i<\delta$. Fix $\beta\leq\alpha$. The inclusion $p^\beta A\subseteq p^\beta B\cap A$ is clear. Conversely, if $a\in p^\beta B\cap A$, choose $i<\delta$ with $a\in A_i$. Then
$a\in p^\beta B\cap A_i=p^\beta A_i\subseteq p^\beta A.$
Thus $A\leq^{p^\alpha}_{\mathrm{iso}}B$.

\item[(5)] Coherence follows from Proposition~\ref{iso_nice}, or directly from \cite[p.~365, Condition~(a)]{fuchs}.

\item[(6)] Let $X\subseteq B\in\zmod$, and let $\kappa:=|\alpha|+\aleph_0$. Let $\Sigma=\{\psi_\beta(x_0)\mid\beta\leq\alpha\}$ 
be the E-nice set given by Proposition~\ref{iso_nice}. We construct an increasing chain $(A_n)_{n<\omega}$ of submodules of $B$. Set $A_0=\langle X\rangle$. Having defined $A_n$, for every $a\in A_n$ and $\beta\leq\alpha$ such that $B\models\psi_\beta(a)$, choose in $B$ a tuple of witnesses for the existential quantifiers of $\psi_\beta(a)$, and let $A_{n+1}$ be the submodule generated by $A_n$ together with all these witnesses. At each stage, at most $|A_n|+\kappa$ tuples, each of length at most $\kappa$, are added. Hence
$|A_n|\leq |X|+\kappa$
for every $n<\omega$.

\smallskip \noindent Let $A:=\bigcup_{n<\omega}A_n$. Then $X\subseteq A\leq B$ and $|A|\leq|X|+\kappa$. If $a\in A$ and $B\models\psi_\beta(a)$, the required witnesses were added at a later stage, so $A\models\psi_\beta(a)$. The converse follows from Proposition~\ref{preserve}. Thus $A\leq_\Sigma B$, and Proposition~\ref{iso_nice} gives $A\leq^{p^\alpha}_{\mathrm{iso}}B$.
\end{enumerate}
\end{proof}

\begin{corollary}
For every ordinal $\alpha$ and every prime $p$, the AEC $(\zmod,\leq^{p^\alpha}_{\mathrm{iso}})$ is stable.
\end{corollary}
\begin{proof}
This follows from Proposition~\ref{iso_nice}, Lemma~\ref{iso_AEC}, and Theorem~\ref{sta-nice}.
\end{proof}

Recall that the \emph{$p$-length} of a $\mathbb Z$-module $A$, denoted by $\operatorname{len}_p(A)$, is the least ordinal $\alpha$ such that $p^{\alpha+1}A=p^\alpha A$.

\begin{proposition}\label{crucial_prop}
Let $p$ be a prime. If $\theta$ is an infinite cardinal, then $
\mathrm{LS}(\zmod,\leq^{p^\theta}_{\mathrm{iso}})=\theta$.
\end{proposition}
\begin{proof}
Lemma~\ref{iso_AEC} gives $
\mathrm{LS}(\zmod,\leq^{p^\theta}_{\mathrm{iso}})\leq\theta$.
Suppose, toward a contradiction, that  $
\mathrm{LS}(\zmod,\leq^{p^\theta}_{\mathrm{iso}}) = \lambda<\theta$.

\smallskip \noindent By \cite[Theorem~10.1.6]{fuchs}, there is a $p$-group $B$ of $p$-length $\theta+1$. Choose $
b\in p^\theta B\setminus p^{\theta+1}B$.
The L\"owenheim--Skolem--Tarski axiom yields $A\leq^{p^\theta}_{\mathrm{iso}}B$ such that $b\in A$ and $|A|\leq\lambda<\theta$. By cardinality considerations, $
\operatorname{len}_p(A)<|A|^+\leq\theta$.
On the other hand, $\operatorname{len}_p(A)\geq\theta+1$. Indeed, if $p^\alpha A=p^{\alpha+1}A$ for some $\alpha<\theta+1$, then an easy induction gives that $p^\theta A=p^{\theta+1}A$. This is impossible, since $
b\in p^\theta B\cap A=p^\theta A$
but $b\notin p^{\theta+1}A$ as $b\notin p^{\theta+1}B$. This contradiction proves the result.
\end{proof}

\begin{theorem}\label{class_size}
The lattice $\mathscr{L}_{\mathbb Z}$ is a proper class. More precisely, for every fixed prime $p$, the relations $
\{\leq^{p^\theta}_{\mathrm{iso}}\mid\theta\in\mathrm{Card}\}$
form a strictly increasing proper-class-sized chain in $\mathscr{L}_{\mathbb Z}$.
\end{theorem}
\begin{proof}
The relations become stronger as $\theta$ increases, and Proposition~\ref{crucial_prop} shows that relations indexed by distinct infinite cardinals have distinct L\"owenheim--Skolem numbers. Hence the chain is strict and proper-class-sized.
\end{proof}

\begin{remark}
Let $\alpha$ be an ordinal and let $p$ be a prime. Define $A\leq^{p^\alpha}_{\mathrm{iso},\mathrm{pp}}B$
if and only if $A\leq^{p^\alpha}_{\mathrm{iso}}B$ and $A\pleq B$.

\smallskip \noindent  All results proved above for $\leq^{p^\alpha}_{\mathrm{iso}}$ extend to $\leq^{p^\alpha}_{\mathrm{iso},\mathrm{pp}}$. In particular, $
\{\leq^{p^\theta}_{\mathrm{iso},\mathrm{pp}}\mid\theta\in\mathrm{Card}\}$
forms a strictly increasing proper-class-sized chain in $\mathscr{L}^{2}_{\mathbb Z}$ and $\mathscr{L}^{2}_{\mathbb Z}$ is a proper class.
\end{remark}

\subsection{Failure of amalgamation}\label{5.3}

We conclude by showing that amalgamation often fails in $\mathscr{L}_{\mathbb Z}$. More precisely, if $\preccurlyeq\in\mathscr{L}_{\mathbb Z}$ is at least as strong as $\pleqq$ and satisfies a suitable syntactic hypothesis, then $(\zmod,\preccurlyeq)$ fails amalgamation. This indicates that the model theory of strong submodel relations in $\mathscr{L}^{2}_{\mathbb Z}$ is  more complicated than that of $(\zmod,\pleq)$.

\begin{definition}
The abstract class $(\rmod,\preccurlyeq)$ is \emph{closed under pushouts} if, for every $A\preccurlyeq B,C$, the  maps of the pushout
$(B\oplus_A C,f:B\to B\oplus_A C,g:C\to B\oplus_A C)$
are strong embeddings.
\end{definition}

\begin{proposition}\label{push-AP}
Let $\preccurlyeq\in\mathscr{L}_{R}$ be positive syntactic.  $(\rmod,\preccurlyeq)$ has the amalgamation property if and only if  $(\rmod,\preccurlyeq)$ is closed under pushouts.
\end{proposition}
\begin{proof}
If $(\rmod,\preccurlyeq)$ is closed under pushouts, then it clearly has amalgamation. Conversely, assume amalgamation, and let $A\preccurlyeq B,C$. Choose an amalgam $D$ and strong embeddings $h_B:B\to D$ and $h_C:C\to D$ satisfying
$h_B\restriction A=h_C\restriction A.$
We show that the canonical map $f:B\to B\oplus_A C$ is a strong embedding; the argument for $g:C\to B\oplus_A C$ is analogous.

\smallskip \noindent Choose $\Sigma\subseteq\Lambda^{\infty,\infty}_\infty$ such that $\preccurlyeq=\leq_\Sigma$. By Remark~\ref{pushout}, $f$ is a monomorphism. It remains to show that, for every $\varphi(\bar x)\in\Sigma$ and every $\bar b\in B^{|\bar x|}$, $B\oplus_A C\models\varphi(f(\bar b))$ if and only if 
$B\models\varphi(\bar b).$

\smallskip \noindent The reverse implication follows from Proposition~\ref{preserve}. For the forward implication, suppose that $B\oplus_A C\models\varphi(f(\bar b))$. By the universal mapping property of the pushout, there is a homomorphism $k:B\oplus_A C\to D$ such that $k\circ f=h_B$ and $k\circ g=h_C$. Proposition~\ref{preserve} gives
$D\models\varphi(h_B(\bar b)).$
Since $h_B$ is a strong embedding, $B\models\varphi(\bar b)$. Thus $f$ is a strong embedding.
\end{proof}

\begin{theorem}\label{no_AP_th}
If $\preccurlyeq\in\mathscr{L}_{\mathbb Z}$ is positive syntactic and lies above $\pleqq$ in $\mathscr{L}_{\mathbb Z}$, then $(\zmod,\preccurlyeq)$ does not have the amalgamation property.
\end{theorem}
\begin{proof}
Let $\preccurlyeq\in\mathscr{L}_{\mathbb Z}$ be positive syntactic and lie above $\pleqq$. We show that $(\zmod,\preccurlyeq)$ is not closed under pushouts; Proposition~\ref{push-AP} then gives the result.

\smallskip \noindent Choose prime numbers $(p_i)_{i<\omega}$ such that
\[
p_{i+1}>p_i^2(\prod_{j=0}^{i-1} p_j)
\]
for every $i<\omega$.  For each $f: \omega \to \mathbb{Z}$, let $\hat{f}: \omega \to \mathbb{Z}/ p_i \mathbb{Z}$  be given by $\hat{f}(i)= f(i)+ p_i \mathbb{Z}$.

Define
\begin{enumerate}[(i)]
\item $M=\prod_{i<\omega} \mathbb Z/p_i\mathbb Z $;
\item
\[
A=\{\hat{f} \mid f: \omega \to \mathbb{Z}\text{ and } (\exists k<\omega)(\forall i>k)\ f(i)=0\};
\]
\item
\[
B=\{\hat{f} \mid f: \omega \to \mathbb{Z}\text{ and } (\exists k<\omega)(\forall i>k)\ f(i+1)=p_i f(i)\};
\]
\item
\[
\begin{split}
C=\{\hat{f} \mid f: \omega \to \mathbb{Z}\text{ and } {}&(\exists k<\omega)(\forall i>k)\\
&\begin{cases}
 f(i)=0 & \text{if $i$ is odd},\\
 f(i+2)=p_ip_{i+1}f(i)\text{ in }\mathbb Z & \text{if $i$ is even}
 \end{cases}\}.
\end{split}
\]
\end{enumerate}

\smallskip \noindent Notice that $A \leq B, C \leq M$. For $n<\omega$, let
\[
A(n)=\{\hat{f} \mid f: \omega \to \mathbb{Z}\text{ and }  (\forall i>n)\ f(i)=0\}.
\]
Then:
\begin{enumerate}[(1)]
\item $A(n)\leq_\oplus B$ and $A(n)\leq_\oplus C$ for every $n<\omega$. Indeed, the intended complements are
\[
D(n)=\{\hat{f} \mid f: \omega \to \mathbb{Z}\text{ , } \hat{f}\in B \text{ and }(\forall i\leq n)\ f(i)=0\}
\]
and
\[
E(n)=\{\hat{f} \mid f: \omega \to \mathbb{Z}\text{ , } \hat{f}\in C \text{ and }(\forall i\leq n)\ f(i)=0\}.
\]
Thus $A(n)\preccurlyeq B,C$.
\item If $n\leq m<\omega$, then $A(n)\leq_\oplus A(m)$, with complement
\[
A'(n)=\{\hat{f} \mid f: \omega \to \mathbb{Z}\text{ , } \hat{f}\in A(m) \text{ and }(\forall i\leq n)\ f(i)=0 \}.
\]
Thus $A(n)\preccurlyeq A(m)$.
\end{enumerate}

\smallskip \noindent Since $A=\bigcup_{n<\omega}A(n)$, smoothness gives $A\preccurlyeq B,C$.

\smallskip \noindent Let $f:B\to B\oplus_A C$ and $g:C\to B\oplus_A C$ be the canonical maps of the pushout  (see Remark \ref{pushout}). We show that $f$ is not a $\pleqq$-embedding and hence is not a $\preccurlyeq$-embedding.

\smallskip \noindent Let $\varphi(x):=\exists z\,\psi(x,z)$, where
\[
\psi(x,z):=
\bigwedge_{n<\omega}
\exists y_0\cdots\exists y_{2n-1}
\begin{pmatrix}
\displaystyle \bigwedge_{j<2n}p_jy_j=0,\\[2pt]
\displaystyle \prod_{j<2n}p_j\mid x-\sum_{j<2n}y_j,\\[2pt]
\displaystyle \prod_{j<2n}p_j\mid z-
\sum_{\substack{j<2n\\ j\text{ even}}}y_j
\end{pmatrix}.
\]
Observe that $\varphi(x)\in\Lambda^{\mathrm{wpp}}_\infty$.

\smallskip \noindent Let $b: \omega \to \mathbb{Z}$ be defined by $b(0)=1$ and
$b(i+1)=p_i b(i)$
for every $i<\omega$. Note that $\hat{b} \in B$ and  we claim that
\[
B\oplus_A C\models\varphi(f(\hat{b}))
\quad\text{but}\quad
B\not\models\varphi(\hat{b}).
\]

\smallskip \noindent Let $c: \omega \to \mathbb{Z}$ be defined by
\[
c(i)=
\begin{cases}
0 & \text{if $i$ is odd},\\
1 & \text{if $i=0$},\\
p_{i-2}p_{i-1}c(i-2) & \text{if $i\geq2$ is even}.
\end{cases}
\]

\smallskip \noindent Note that $\hat{c} \in C$.

\smallskip \noindent Fix $n<\omega$. For every $j<2n$, let $a_j:\omega\to\mathbb{Z}$ be given by $a_j(j)=b(j)$ and $a_j(i)=0$ if $i \neq j$. Note that $\widehat{a_j} \in A$ for every $j < 2n$. 

\smallskip \noindent We have that  $B\models\bigwedge_{j<2n}p_j\widehat{a_j}=0$. Let $b': \omega \to \mathbb{Z}$ be given by
\[
b'(i)=
\begin{cases}
0 & \text{if $i<2n$},\\[2pt]
\displaystyle\frac{b(i)}{p_0\cdots p_{2n-1}} & \text{if $i\geq2n$}.
\end{cases}
\]
Note that $\hat{b'} \in B$ and 
\[
B\models
\prod_{j<2n}p_j \hat{b'}=\hat{b}-\sum_{j<2n}\widehat{a_j}.
\]
Similarly,
\[
C\models
\prod_{j<2n}p_j\mid
\hat{c}-\sum_{\substack{j<2n\\j\text{ even}}}\widehat{a_j}.
\]
Using Proposition~\ref{preserve} and the identity $f\restriction A=g\restriction A$, we conclude that
$
B\oplus_A C\models\psi(f(\hat{b}),g(\hat{c})),
$
and hence $B\oplus_A C\models\varphi(f(\hat{b}))$.

\smallskip \noindent It remains to show that $B\not\models\varphi(\hat{b})$. Suppose otherwise, let $e:\omega\to\mathbb{Z}$ such that $\hat{e}\in B$ witnesses the formula, i.e.,  $B\models\psi(\hat{b},\hat{e})$. Then there is $k<\omega$ such that
$e(i+1)=p_i e(i)$
for every $i>k$. Let $k' <\omega$ with $k' > k$ such that $|e(k)| < p_{k'}$. Then using that $p_{i+1}>p_i^2(\prod_{j=0}^{i-1} p_j)$
for every $i<\omega$, it follows that for every $i> k'$ we have both:
\begin{equation}
e(i+1)=p_i e(i) \text{ and } |e(i)| < p_i
\end{equation}

\smallskip \noindent Choose $n$ large enough such that there is an odd integer $m$ with
\[
k' <m
\quad\text{and}\quad
m+1<2n.
\]  

\smallskip \noindent Let $(\widehat{a_j})_{j<2n}$ be witnesses to the $n$th conjunct of $\psi(\hat{b},\hat{e})$.

\smallskip \noindent For every $j < 2n$, it follows from $B \models p_j \widehat{a_j}=0$ and the fact that the primes $p_i$'s are distinct,  that $p_i \mid a_j(i)$ for every $i \neq j$. Hence $\widehat{a_j}(i) = 0_{\mathbb Z/p_i\mathbb Z}$ for every $i \neq j$. 

\smallskip \noindent For $j < 2n$, the first divisibility condition then yields \begin{equation}
p_j \mid b(j) - a_j(j).
\end{equation}
The second divisibility condition then yields
\begin{equation}
\begin{cases}
p_j \mid e(j) & \text{if $j<2n$ is odd},\\
p_j \mid e(j) - a_j(j) & \text{if $j<2n$ is even}.
\end{cases}
\end{equation}

\smallskip \noindent In particular, for $m < 2n$, which is odd, we have that $p_m \mid e(m)$ by Equation (3). Since $m > k'$ we also have that $|e(m)| < p_m$ by Equation (1). Hence $e(m)=0$. As $e(m+1)=p_me(m)$ by Equation (1), we conclude that $e(m+1)=0$.

\smallskip \noindent Since $m+1 < 2n$ is even, it follows that \[ p_{m+1} \mid e(m+1) - a_{m+1}(m+1)= -a_{m+1}(m+1)\] from Equation (3) and the last equation of the previous paragraph. Then it follows that $p_{m+1} \mid b(m+1)$ by Equation (2). This is a contradiction as an easy induction yields that $0 < b(m+1) < p_{m+1}$. \end{proof}

\begin{remark} 
Observe that $f: B \to B \oplus_A C$ of the proof of Theorem \ref{no_AP_th} is an explicit example of a pure embedding which is not a $\pleqq$-embedding. See Remark \ref{rmk-compare} for another example.
\end{remark}

\newpage 
\appendix
\section{Depiction of  $\mathscr{L}_{\mathbb Z}$}\label{app-A}
The following figure depicts our current understanding of $\mathscr{L}_{\mathbb Z}$: 

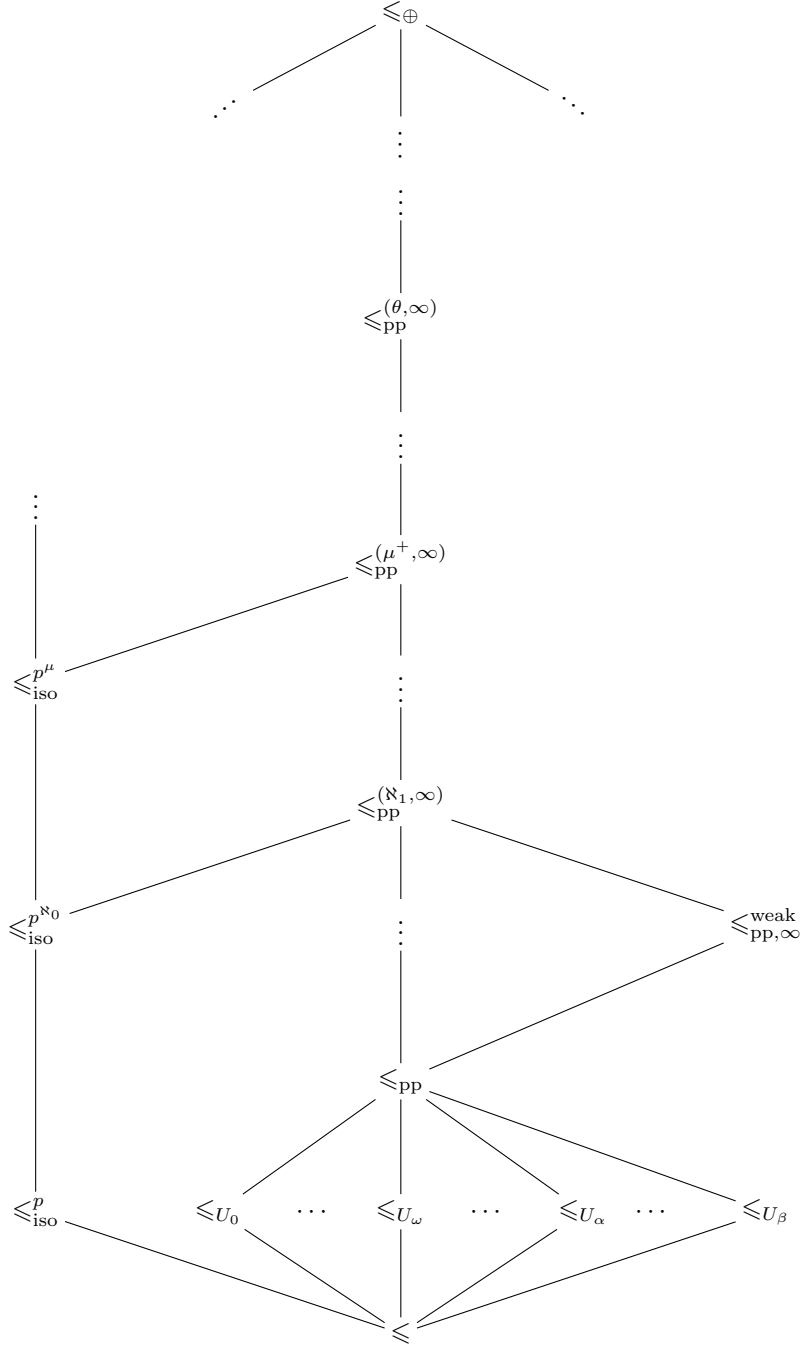
\begin{figure}[H]
\centering
\begin{tikzpicture}[scale=1.15,every node/.style={inner sep=2pt, font=\normalsize},]
\node (leq) at (0,0) {$\leq$};

\node (leq0)   at (-4.2,1.4) {$\leq_{\mrm{iso}}^p$};
\node (leqn)   at (-2.1,1.4) {$\leq_{U_0}$};
\node (cdots2) at (-1.0,1.4) {$\cdots$};
\node (leqw)   at ( 0,  1.4) {$\leq_{U_\omega}$};
\node (cdots3) at ( 1.0,1.4) {$\cdots$};
\node (leqa)   at ( 2.1,1.4) {$\leq_{U_\alpha}$};
\node (cdots4) at ( 2.9,1.4) {$\cdots$};
\node (leqb)   at ( 4.2,1.4) {$\leq_{U_\beta}$};

\draw (leq) -- (leq0);
\draw (leq) -- (leqn);
\draw (leq) -- (leqw);
\draw (leq) -- (leqa);
\draw (leq) -- (leqb);

\node (leqpp) at (0,2.9) {$\leq_{\mathrm{pp}}$};
\draw (leqn) -- (leqpp);
\draw (leqw) -- (leqpp);
\draw (leqa) -- (leqpp);
\draw (leqb) -- (leqpp);

\node (leqisoalpha) at (-4.2,4.7) {$\leq^{p^{\aleph_0}}_{\mathrm{iso}}$};
\node (leqisobeta)  at (-4.2,7.5) {$\leq^{p^\mu}_{\mathrm{iso}}$};
\node (cdotsL)      at (-4.2,9.6) {$\vdots$};

\draw (leq0)        -- (leqisoalpha);
\draw (leqisoalpha) -- (leqisobeta);
\draw (leqisobeta)  -- (cdotsL);

\node (leqww1) at (4.2,4.7) {$\pleqq$};
\draw (leqpp) -- (leqww1);

\node (cdotsC0) at (0,4.7) {$\vdots$};
\node (leqpptheta) at (0,6.1) {$\leq^{(\aleph_1,\infty)}_{\mathrm{pp}}$};
\node (cdotsC1) at (0,7.5) {$\vdots$};
\node (leqppcard) at (0,8.9){$\leq^{(\mu^+,\infty)}_{\mathrm{pp}}$};
\node (cdotsC2) at (0,10.3) {$\vdots$};
\node (leqppmu) at (0,11.7) {$\leq^{(\theta,\infty)}_{\mathrm{pp}}$};
\node (cdotsC3) at (0,13.1) {$\vdots$};

\draw (leqpp)     -- (cdotsC0);
\draw (cdotsC0)   -- (leqpptheta);
\draw (leqpptheta) -- (cdotsC1);
\draw (cdotsC1)   -- (leqppcard);
\draw (leqppcard) -- (cdotsC2);
\draw (cdotsC2)   -- (leqppmu);
\draw (leqppmu)   -- (cdotsC3);
\draw (leqww1) -- (leqpptheta);
\draw (leqisoalpha) -- (leqpptheta);
\draw (leqisobeta) -- (leqppcard);

\node (leqoplus) at (0,15.2) {$\leq_\oplus$};
\draw (leqoplus) -- (-1.7,14.3);
\node[rotate=40] at (-2.0,14.1) {$\cdots$};
\draw (leqoplus) -- (0,14.0);
\node[rotate=90] at (0,13.7) {$\cdots$};
\draw (leqoplus) -- (1.7,14.3);
\node[rotate=-40] at (2.0,14.1) {$\cdots$};
\end{tikzpicture}

\caption{Strong submodel relations on $\mathbb{Z}\text{-}\mathrm{Mod}$.}
\label{fig:lattice-zmod}
\end{figure}

\end{document}